\documentclass[letterpaper,11pt]{amsart}
\usepackage{amsmath}
\usepackage{amssymb}
\usepackage[dvipdfmx]{graphicx}
\usepackage{amscd}
\usepackage{amsthm}
\usepackage{fullpage}
\usepackage{mathrsfs}
\usepackage{setspace}
\usepackage{mathtools}
\usepackage{xcolor}
\usepackage[all]{xy}
\usepackage{url}
\usepackage{bbm}
\usepackage[colorlinks=true,linkcolor=blue,breaklinks,pagebackref]{hyperref}
\usepackage[poorman]{cleveref}
\usepackage[lite,abbrev,alphabetic]{amsrefs}

\DefineSimpleKey{bib}{myurl}
\newcommand\myurl[1]{\url{#1}}
\BibSpec{webpage}{%
  +{}{\PrintAuthors} {author}
  +{,}{ \textit} {title}
  +{}{ \parenthesize} {date}
  +{,}{ \myurl} {myurl}
  +{,}{ } {note}
  +{.}{ } {transition}
}

\theoremstyle{definition}
\newtheorem{definition}{Definition}[section]
\newtheorem{example}[definition]{Example}

\theoremstyle{plain}
\newtheorem{theorem}[definition]{Theorem}
\newtheorem{prop}[definition]{Proposition}
\newtheorem{lemma}[definition]{Lemma}
\newtheorem{cor}[definition]{Corollary}
\newtheorem{conj}[definition]{Conjecture}

\newtheorem{remark}[definition]{Remark}
\newtheorem{condition}[definition]{Condition}

\crefname{theorem}{Theorem}{Theorems}
\crefname{prop}{Proposition}{Propositions}
\crefname{lemma}{Lemma}{Lemmas}
\crefname{cor}{Corollary}{Corollaries}
\crefname{conj}{Conjecture}{Conjectures}
\crefname{hypo}{Hypothesis}{Hypotheses}
\crefname{remark}{Remark}{Remarks}
\crefname{condition}{Condition}{Conditions}
\crefname{example}{Example}{Examples}

\def\C{\mathbb{C}}
\def\R{\mathbb{R}}
\def\Q{\mathbb{Q}}
\def\Z{\mathbb{Z}}
\def\A{\mathbb{A}}

\def\GL{\mathrm{GL}}
\def\PGL{\mathrm{PGL}}

\def\M{\mathrm{M}}

\def\d{\,\mathrm{d}}
\def\ds{\displaystyle}
\def\bs{\backslash}

\def\new{\mathrm{new}}

\definecolor{gr}{rgb}{0.2,0.7,0}
\def\fin{\mathrm{fin}}

\DeclareMathOperator{\diag}{diag} 
\DeclareMathOperator{\vol}{vol}
\DeclareMathOperator{\ord}{ord}
\DeclareMathOperator{\mass}{mass}
\DeclareMathOperator{\Gal}{Gal}
\DeclareMathOperator{\Hom}{Hom}
\DeclareMathOperator{\Emb}{Emb}
\DeclareMathOperator{\Ad}{Ad}
\DeclareMathOperator{\disc}{disc}
\DeclareMathOperator{\level}{level}
\DeclareMathOperator{\sgn}{sgn}
\DeclareMathOperator{\re}{Re}
\DeclareMathOperator{\Nm}{Nm}
\DeclareMathOperator{\Tr}{Tr}
\DeclareMathOperator{\Cl}{Cl}
\DeclareMathOperator{\Typ}{Typ}
\DeclareMathOperator{\St}{St}
\DeclareMathOperator{\AL}{AL}

\newcommand{\fa}{\mathfrak{a}}

\newcommand{\fH}{\mathfrak{H}}

\newcommand{\fo}{\mathfrak{o}}

\newcommand{\fP}{\mathfrak{P}}

\newcommand{\bE}{\mathbb{E}}

\newcommand{\bP}{\mathbb{P}}

\newcommand{\cA}{\mathcal{A}}

\newcommand{\cE}{\mathcal{E}}

\newcommand{\cO}{\mathcal{O}}
\newcommand{\cP}{\mathcal{P}}

\newcommand{\cS}{\mathcal{S}}

\newcommand{\cW}{\mathcal{W}}

\numberwithin{equation}{section}

\def\MR#1{\quad \href{http://www.ams.org/mathscinet-getitem?mr=#1}{MR#1}}

\title{Distribution of toric periods of modular forms on definite quaternion algebras}

\author{Miyu Suzuki}

\author{Satoshi Wakatsuki}

\author{Shun'ichi Yokoyama}

\address{Miyu Suzuki \\
Faculty of Mathematics and Physics, Institute of Science and Engineering\\
Kanazawa University\\
Kakumamachi, Kanazawa, Ishikawa, 920-1192}
\email{miyu-suzuki@staff.kanazawa-u.ac.jp}

\address{Satoshi Wakatsuki \\
Faculty of Mathematics and Physics, Institute of Science and Engineering\\
Kanazawa University\\
Kakumamachi, Kanazawa, Ishikawa, 920-1192}
\email{wakatsuk@staff.kanazawa-u.ac.jp}

\address{Shun'ichi Yokoyama \\
Department of Mathematical Sciences\\
Graduate School of Science, Tokyo Metropolitan University\\
1-1 Minami-Osawa, Hachioji-shi, Tokyo, 192-0397}
\email{s-yokoyama@tmu.ac.jp}

\begin{document}

\maketitle

\begin{abstract}
Let $D$ be a definite quaternion algebra over $\Q$ and $\cO$ an Eichler order in $D$ of square-free level.
We study distribution of the toric periods of algebraic modular forms of level $\cO$.
We focus on two problems: non-vanishing and sign changes.
Firstly,  under certain conditions on $\cO$,  we prove the non-vanishing of the toric periods for positive proportion of imaginary quadratic fields.
This improves the known lower bounds toward Goldfeld's conjecture in some cases and provides evidence for similar non-vanishing conjectures for central values of twisted automorphic $L$-functions. 
Secondly,  we show that the sequence of toric periods has infinitely many sign changes.
This proves the sign changes of the Fourier coefficients $\{a(n)\}_n$ of weight $\frac32$ modular forms,  where $n$ ranges over fundamental discriminants.
In the final section,  we present numerical experiments in some cases and formulate several conjectures based on them. 
\end{abstract}

\tableofcontents


\section{Introduction}
Let $D$ be a quaternion algebra over $\Q$ and $E$ a quadratic field which embeds in $D$.
For a cuspidal automorphic form $\phi$ on $G_\A:=(D\otimes_\Q\A)^\times/\A^\times$,  the \textit{toric period} $\cP_E(\phi)$ is defined as an integral
    \[
    \cP_E(\phi)=\int_{\A^\times E^\times\bs(E\otimes_\Q\A)^\times}\phi(h)\d h,
    \]
where $\A$ is the ring of ad\`eles of $\Q$.
Suppose that $D$ is definite,  \textit{i.e.}\,$D\otimes_\Q\R$ is the Hamilton's quaternion. 
When $\phi$ is right invariant under some open compact subgroup of $G_\A$ obtained from an order in $D$,  it is called an \textit{algebraic modular form}.
In this paper,  we study the distribution of toric periods $\{\cP_E(\phi)\}_E$ of algebraic modular forms $\phi$,  where $E$ runs over suitable set of quadratic fields.


\subsection{Results}
\label{subsec:Results}

Let $G=\PGL_1(D)$ be the algebraic group over $\Q$ such that for a $\Q$-algebra $R$,  the group of $R$-points $G_R$ is $(D\otimes_\Q R)^\times/R^\times$.
Take an Eichler order $\cO$ in $D$ of square-free level.
Let $S_\cO$ be the finite set of places of $\Q$ consisting of the real place and all primes $p$ at which $\cO_p:=\cO\otimes_\Z\Z_p$ is not isomorphic to $\M_2(\Z_p)$,  the algebra of 2 by 2 matrices over $\Z_p$.
For a prime $p$,  let $N_p$ be the normalizer in $G_{\Q_p}$ of $\overline{(\cO_p^\times)}$,  where $\overline{\,\cdot\,}$ denotes the image under the projection $(D\otimes\Q_p)^\times\rightarrow G_{\Q_p}$.
For the real place $v=\infty$,  set $N_\infty=G_\R$.
Let $\cS_N(\cO)$ be the space of cusp forms on $G_\A$ which is right invariant under $N=\prod_v N_v$.
As for the holomorphic modular forms,  we can define the Hecke operator $T_p$ on $\cS_N(\cO)$ for each prime $p$.
A simultaneous eigenvector of $\{T_p\}_p$ is called a \textit{Hecke eigenform}.
If an Eichler order $\cO'$ in $D$ contains $\cO$,  we may regard $\cS_N(\cO')$ as a subspace of $\cS_N(\cO)$.
The orthogonal complement of $\sum_{\cO\subsetneqq\cO'}\cS_N(\cO')$ in $\cS_N(\cO)$ is denoted by $\cS_N^\new(\cO)$.

We focus on two aspects of the distribution of toric periods: non-vanishing and sign changes.
First we state the non-vanishing result.
Let $\phi\in\cS_N^\new(\cO)$ be a Hecke eigenform,  $\pi$ the irreducible cuspidal automorphic representation of $G_\A$ generated by $\phi$ and $\pi'$ its Jacquet-Langlands transfer to $\PGL_2(\A)$.
For each $v\in S_\cO$,  we take a quadratic \'etale algebra $\cE_v$ over $\Q_v$ which embeds in $D_v:=D\otimes_\Q\Q_v$ so that the ramification set of $D$ coincides with the set of places $v$ at which $\varepsilon(\pi'_v; \cE_v)=-1$ (see \cref{thm:Waldspurger}).
Let $X$ be the set of all quadratic fields and $X(\{\cE_v\}_{v\in S_\cO})$ the set of $E\in X$ satisfying $E\otimes_\Q\Q_v\simeq\cE_v$ for all $v\in S_\cO$.
The discriminant of $E\in X$ is denoted by $\Delta_E$.

For the non-vanishing result,  we make two assumptions on $\cO$.
One is that $\cS_N(\cO)$ is spanned by a single Galois orbit of Hecke eigenforms.
The other one is that the numerator of the total mass $\mass(\cO)=\vol(G_\Q\bs G_\A)$ is divisible by $3$.
See \cref{condition:non-vanish} for details.

\begin{theorem}[\cref{thm:PeriodGoldfeld}]\label{thm:intro1}
Suppose that \cref{condition:non-vanish} holds.
Then  
    \[
    \#\{E\in X(\{\cE_v\}_{v\in S_\cO}) \mid |\Delta_E|<x, \ 
    \cP_E(\phi)\neq0\}\gg x
    \]
when $x\to\infty$.
\end{theorem}

By the celebrated result of Waldspurger (\cref{thm:Waldspurger}),  non-vanishing of toric periods is related to non-vanishing of central values of automorphic $L$-functions. 
We can obtain an explicit lower bound for the proportion of the non-vanishing of central values of automorphic $L$-functions.

\begin{theorem}[\cref{thm:AutomGoldfeld}]\label{thm:intro2}
Suppose that \cref{condition:non-vanish} holds.
Then
    \[
    \lim_{x\to\infty}
    \frac{\#\{E\in X\mid -x< \Delta_E<0,  \ L(\tfrac12,  \pi'\otimes\eta_E)\neq0\}}
    {\#\{E\in X\mid -x<\Delta_E<0\}} \
    \geq \ \frac12 \ \prod_{v\in S_\cO,  v<\infty}n_v,
    \]
where $\eta_E$ is the quadratic character on $\A^\times/\Q^\times$ attached to $E$,  $L(s,  \pi'\otimes\eta_E)$ is the standard $L$-function of $\pi'\otimes\eta_E$,  $n_p=\frac{p+2}{2(p+1)}$ if $p\neq2$ and $n_2=\frac{1}{24}$.
\end{theorem}

A conjecture of Goldfeld \cite{Goldfeld} asserts that for an elliptic curve over $\Q$,  the central $L$-values of 50\% of its quadratic twists do not vanish.
If all the Hecke eigenvalues of $\phi$ are in $\Q$,  then the finite part of the standard $L$-function $L(s-\frac12,  \pi')$ of $\pi'$ coincides with the $L$-function $L(s,  C)$ of an elliptic curve $C$ over $\Q$.
In that case,  $L(\frac12,  \pi'\otimes\eta_E)\neq0$ is equivalent to $L(1,  C_E)\neq0$,  where $C_E$ is the quadratic twist of $C$ by $\Delta_E$.
From this point of view,  we formulate a natural generalization (\cref{conj:AutomGoldfeld}) of Goldfeld's conjecture for automorphic $L$-functions of $\PGL_2(\A)$ and related conjectures for toric periods (\cref{conj:WeakPeriodGoldfeld} and \cref{conj:PeriodGoldfeld}). 
The above two theorems provide evidence for these conjectures.

For our result on sign changes,  assume that $\cO$ is a maximal order.
By replacing $\cO$ with another maximal order which is locally isomorphic to $\cO$ if necessary,  we may assume there exists an embedding $\iota\,\colon E\hookrightarrow D$ such that $\iota(E)\cap\cO=\iota(\fo_E)$.
Here,  $\fo_E$ denotes the ring of integers of $E$.
Such $\iota$ is called an \textit{optimal embedding}.
See the argument in \cref{subsec:Dependence} for details.
Using such an embedding,  one can rewrite the toric period $\cP_E(\phi)$ as a finite sum $\fP_E(\phi)$ over the ideal class group of $E$.

Recall that $\phi\in\cS_N^\new(\cO)$ is a Hecke eigenform.
Let $F_\pi$ be the number field generated by the all Hecke eigenvalues of $\phi$,  $\fo_\pi$ its ring of integers. 
Note that $F_\pi$ is totally real.
We fix a $\Z$-basis $\{v_i\}_i$ of $\fo_\pi$ and write the expansion of $x\in\fo_\pi$ as $x=\sum_{i=1}^{[F_\pi:\Q]}x^{(i)}v_i$ with $x^{(i)}\in\Z$.
By multiplying a non-zero constant,  we may assume that $\phi$ takes its values in $\fo_\pi$ and so does $\fP_E(\phi)$.

\begin{theorem}[\cref{thm:sign_change}]\label{thm:intro3}
Keep the above assumption.
\begin{itemize}
\item[(1)] We fix an embedding $F_\pi \hookrightarrow \R$ and regard $\fP_E(\phi)$ as a real number.
Then the sequence $\{\fP_E(\phi)\in\R\}_E$ has infinitely many sign changes.
\item[(2)] The sequence $\{\fP_E(\phi)^{(i)}\in\Z\}_E$ has infinitely many sign changes for at least one $i$.
\end{itemize}
\end{theorem}

As a refinement of this theorem,  we will formulate a conjecture that for $z\in\fo_\pi$,  the probability that $\fP_E(\phi)=z$ and $\fP_E(\phi)=-z$ are ``approximately equal'' (\cref{conj:symmetry}).
This conjecture originates from numerical experiments by using \texttt{Magma} \cite{BCP}.
In \cref{sec:Conjectures},  we provide some graphs which visualize the number of quadratic fields $E$ such that $\fP_E(\phi)=z$ in several cases.


\subsection{Methods}
\label{subsec:Methods}

There are three key ingredients in the proof of \cref{thm:intro1} and \cref{thm:intro2}.
The first one is the congruence result of \cite{Martin} on the existence of $\varphi\in\cS_N(\cO)$ which takes values in $1+p\Z$ for some odd prime $p$.
Since we assume that $\cS_N(\cO)$ is spanned by a single Galois orbit of Hecke eigenforms,  $\fP_E(\phi)\neq0$ holds provided that the class number of $E$ is not divisible by $p$.

The second one is a lower bound for the number of imaginary quadratic fields with the class number coprime to $3$.
For our purpose,  \cite[Proposition 9.3]{KL} is sufficient.

The last one is the existence of optimal embeddings. 
We use the fact that under a reasonable condition,  a quadratic field $E$ which embeds in $D$ has an optimal embedding if $|\Delta_E|$ is sufficiently large (\cref{lem:Duke}).
It is a consequence of a variant of Duke's theorem \cite{Duke} and the subconvex bound for central values of twisted $L$-functions \cite{BH}.

The proof of \cref{thm:intro3} is based on an analysis of the Hecke $L$-series $D(s,  h)$,  the Mellin transform a holomorphic modular form $h$ of weight $\frac32$.
When $h=\cW(\phi)$ is the classical Waldspurger's lift of $\phi\in\cS_N^\new(\cO)$,  an explicit computation of $D(s,  h)$ is carried out in \cref{subsec:Hecke $L$-series} to get the following.

\begin{prop}[\cref{prop:Hecke}]\label{prop:intro}
For a Hecke eigenform $\phi\in\cS_N^\new(\cO)$,
    \[
    D(s,  \cW(\phi))=(2\pi)^{-s}\Gamma(s)L_\fin(2s-\tfrac12,  \pi)
    \sum_E
    \frac{c(E)\fP_E(\phi)}{L^{S_\cO}(2s,  \eta_E)|\Delta_E|^s}.
    \]
Here,  $L_\fin(s,  \pi)$ is the finite part of the standard $L$-function of $\pi$,  $E$ ranges over all quadratic fields which embed in $D$,  $c(E)$ is a certain power of $2$ and $L^{S_\cO}(s,  \eta_E)$ is the partial Dirichlet $L$-function.
\end{prop}

Suppose that there are only finitely many sign changes.
We may assume $\fP_E(\phi)>0$ for almost all $E$.  
It follows from the above explicit formula for $D(s,  \cW(\phi))$ combined with the mean value formula for $\{|\fP_E(\phi)|^2\}_E$ obtained in \cite{SWY} that $D(s,  \cW(\phi))$ has a pole at $s=s_0$ for some $s_0\geq\frac34$.
On the other hand,  one can see that $D(s,  \cW(\phi))$ is entire from its functional equation and the Hecke bound for the Fourier coefficients of half-integral weight modular forms.
This is a contradiction.

The classical Waldspurger's lift of an algebraic modular form is introduced by Gross \cite{Gross}. 
B\"ocherer and Schulze-Pillot \cite{BSP1} proved an explicit relation between the square of the Fourier coefficients of $\cW(\phi)$ and the central $L$-values of the Jacquet-Langlands transfer of $\phi$.
Their result is an explicit version of the Waldspurger's formula \cref{thm:Waldspurger} (2),  which relates the central $L$-values with the square of the toric periods of $\phi$.
Hence one naturally expects an equality between the Fourier coefficients of $\cW(\phi)$ and the toric periods $\fP_E(\phi)$.
We prove that the $|\Delta_E|$-th Fourier coefficient $a_\phi(|\Delta_E|)$ of $\cW(\phi)$ equals $c(E)\fP_E(\phi)$ (see \cref{cor:Fourier}).
This plays an important role in the proof of \cref{prop:intro}.


\subsection{Related works}
\label{subsec:Related}

The non-vanishing result of this paper is inspired by \cite{Martin}.
The main new ingredient in our proof is the existence of optimal embeddings.
As already noted,  under a mild condition,  we establish the existence of optimal embeddings for all but finitely many quadratic fields which embed in $D$.
It should be also mentioned that we take care of the local-global compatibility of the number of conjugacy classes of optimal embeddings (\cref{subsec:Action}).
As a consequence,  we see that a quadratic field $E$ which embeds in $D$ has an optimal embedding for some Eichler orders in a fixed local isomorphism class if $E$ splits or ramifies at some finite number of places.
These arguments enable us to treat not only maximal orders but also Eichler orders.

Our research on sign changes of toric periods should be compared with those on Fourier coefficients $a(n)$ of half-integral weight modular forms since $c(E)\fP_E(\phi)=a_\phi(|\Delta_E|)$.
The study of the sign change problem of the Fourier coefficients of half-integral modular forms is initiated by Brunier and Kohnen \cite{BK}.  
They proved that the sequence $\{a(tn^2)\}_n$ for a square-free positive integer $t$ has infinitely many sign changes if $a(t)\neq0$.
For recent progress in this direction we refer to \cite{KLW},  \cite{HK},   \cite{IW} and \cite{Mezroui}.
The investigation of the sequence $\{a(tn^2)\}_n$ is motivated by the fact that it determines integral weight modular forms under the Shimura correspondence.

The sign change problem of $\{a(t)\}_t$,  the sequence restricted to square-free indices,  is studied by \cite{HKKL} and \cite{LRW}. 
They treat half-integral weight modular forms on $\Gamma_0(4)$.
The case of general level is studied quite recently by \cite{LR}.   
They proved that the sequence $\{a(|\Delta_E|)\}_E$ has infinitely many sign changes when $E$ runs through real or imaginary quadratic fields.
However,  it seems that the case of weight $\frac32$ modular forms is excluded in their result.
\cref{thm:intro3} focuses on the remaining case: the case of weight $\frac32$ modular forms of general level.

\vspace{2mm}
\noindent\textbf{Data Availability Statement}
The data that support the findings of this study are available in \url{http://wakatsuki.w3.kanazawa-u.ac.jp/Figures.html}.

\vspace{2mm}
\noindent\textbf{Acknowledgments.} 
The authors thank Tamotsu Ikeda,  Toshiki Matsusaka,  Masataka Chida and Kimball Martin for helpful discussions and valuable comments. 
The authors also thank Siegfried B\"{o}cherer,  Rainer Schulze-Pillot and Winfried Kohnen for answering many questions.
M.S. was partially supported by Grant-in-Aid for JSPS Fellows No.20J00434.
S.W. was partially supported by JSPS Grant-in-Aid for Scientific Research (C) No.18K03235 and (B) No.21H00972.
S.Y. was partially supported by JSPS Grant-in-Aid for
Scientific Research (C) No.20K03537.


\section{Preliminaries}
\label{sec:Preliminaries}


\subsection{General notation}
\label{subsec:General notation}

The cardinality of a finite set $A$ is denoted by $|A|$ or $\#A$.

Throughout the manuscript,  we keep the following setup.
We denote by $\A$ the ring of ad\`eles of $\Q$.
Let $\A_f$ denote the ring of finite ad\`eles.
Finite places of $\Q$ are identified with primes.
We write the real place of $\Q$ as $\infty$.
Let $\Q_v$ denote the completion of $\Q$ at a place $v$.
The completed Dedekind zeta function of $\Q$ is denoted by $\zeta(s)$, which is the Euler product $\prod_v\zeta_v(s)$ of local factors,  where $\zeta_p(s)=(1-p^{-s})^{-1}$ for a prime $p$ and $\zeta_\infty(s)=\pi^{-\tfrac{s}{2}}\Gamma(\tfrac{s}{2})$.

For a quadratic field $E$,  the fundamental discriminant is denoted by $\Delta_E$.
Let $\A_E$ (resp.\,$\A_{E,  f}$) be the ring of (resp.\,finite) ad\`eles of $E$ and $\eta_E=\otimes_v\eta_{E,  v}$ the quadratic character of $\A^\times/\Q^\times$ attached to $E$.
The completed Hecke $L$-function of $\eta_E$ is denoted by $L(s,  \eta_E)$.
For each place $v$,  set $E_v=E\otimes_\Q\Q_v$.
Following the usual convention,  we use the notation $(\frac{\cdot}{\cdot})$ for the Legendre symbol.

Let $D$ be a quaternion algebra over $\Q$. 
We denote by $\disc(D)$ its discriminant. 
For a place $v$ of $\Q$,  set $D_v=D\otimes_\Q\Q_v$.
We also set $D_\A=D\otimes_\Q\A$ and $D_{\A_f}=D\otimes_\Q\A_f$.
Let $X(D)$ be the set of quadratic fields which embed in $D$.
For $E\in X(D)$,  we write the set of embeddings $E \hookrightarrow D$ by $\Emb(E,  D)$.
By the Skolem-Noether theorem,  $D^\times$ acts transitively on $\Emb(E,  D)$ by conjugation.
Let $X$ denote the set of all quadratic fields.
Unless otherwise mentioned,  we assume that $D$ is definite. 
In that case,  $D_\infty=D\otimes_\Q\R$ is isomorphic to the Hamilton's quaternion and $X(D)$ consists of imaginary quadratic fields.
Similarly,  let $X(D_v)$ be the set of quadratic \'etale algebras over $\Q_v$ which embeds in $D_v$.
When $D_v=\M_2(F_v)$ is the algebra of 2 by 2 matrices,  we write it as $X_v$,  the set of all quadratic \'etale algebras over $\Q_v$.

We denote by $G$ the algebraic group over $\Q$ such that $G(R)=(D\otimes_\Q R)^\times/R^\times$ for a $\Q$-algebra $R$.
To simplify the notation,  we write the group of $R$-rational points $G(R)$ as $G_R$ for any $\Q$-algebra $R$.
For a place $v$ of $\Q$,  set $G_v=G(\Q_v)$.
We use similar notation for subgroups defined over $\Q$.
The projection map $D^\times\rightarrow G$ is written as $x\mapsto\bar{x}$.
We denote the image of a subset $Y$ of $D^\times$ by $\overline{Y}$.

Let $\cO$ be an Eichler order in $D$.
Denote its level by $\level(\cO)$ and we assume that $\level(\cO)$ is square-free. 
The discriminant of $\cO$ is defined as $\disc(\cO)=\disc(D)\level(\cO)$.
For a finite place $v$,  set $\cO_v=\cO\otimes_\Z\Z_v$.
Let $S_\cO$ be the finite set of the real place $\infty$ and all finite places $v$ at which $\cO_v\not\simeq\M_2(\Z_v)$.
Equivalently,  $S_\cO$ is the set of the real place and the prime factors of $\disc(\cO)$.

For each place $v$,  let $K_v$ be the open compact subgroup of $G_v$ given by $K_v=\overline{(\cO_v^\times)}$ if $v$ is finite and $K_\infty=G_\R$.
Then $K=\prod_vK_v$ is an open compact subgroup of $G_\A$.
We normalize the Haar measure on $G_\A$ so that $\vol(K)=1$.
For each place $v$,  let $N_v$ be the normalizer of $K_v$ in $G_v$.
Set $N=\prod_v N_v$.
Note that $N_v=K_v$ for $v\not\in S_\cO$ and $[N_v:K_v]=2$ otherwise.
Hence we have $N/K\simeq\prod_{p\mid \disc(\cO)}\Z/2\Z$ (\cite[\S\,23]{Voight}).
In particular,  $[N:K]$ is a power of 2.
Emphasizing the dependence on $\cO$,  we will sometimes write $K_\cO$ and $N_\cO$ in place of $K$ and $N$,  respectively.


\subsection{Algebraic modular forms}
\label{subsec:Algebraic modular}

We summarize necessary facts and notation about algebraic modular forms on $G_\A$.
Most of the material here is well-known.
See \cite{Voight},  for example.

Suppose that $D$ is definite.
A complex valued function $\phi$ on $G_\A$ which satisfies
    \[
    \phi(\gamma gk)=\phi(g), \qquad \gamma\in G_\Q,  \, g\in G_\A,  \, k\in K_\cO
    \]
is called an \textit{algebraic modular form} of level $\cO$.
Let $\cA(\cO)$ be the space of such functions.
The group $G_\A$ acts on the sum of these spaces $\sum_\cO\cA(\cO)$ by the right translation $R_G$,  where $\cO$ runs through (Eichler) orders in $D$.

Set $h_\cO=|G_\Q\bs G_\A/K_\cO|$ and fix a set of representatives $\{x_i\}_{i=1}^{h_\cO}$ of cosets in $G_\Q\bs G_\A/K_\cO$.
Since $K_\infty=G_\R$,  we can and will take $x_i$'s from $G_{\A_f}$.
Note that $\dim_\C\cA(\cO)=h_\cO$.
For $x\in G_\A$ set $w(x)=|G_\Q\cap xKx^{-1}|$ so that we have $\vol(G_\Q\bs G_\A)=\sum_{i=1}^{h_\cO}w(x_i)^{-1}$.
Write this value as $\mass(\cO)$.
We define the inner product $(\ , \ )$ on $\cA(\cO)$ by
    \[
    (\phi_1,  \phi_2)=\sum_{i=1}^{h_\cO}w(x_i)^{-1}\phi_1(x_i)\overline{\phi_2(x_i)},
    \qquad \phi_1,  \, \phi_2 \in\cA(\cO).
    \]
The symbol $\langle \ ,  \ \rangle$ denotes the Petersson inner product on $L^2(G_\Q\bs G_\A)$ with respect to 2 times the Tamagawa measure on $G_\A$. 
Then we have $(\ , \ )=4^{-1}\mass(\cO)\langle \ ,  \ \rangle$.
Let $\cS(\cO)$ denote the orthogonal complement of the space of constant functions in $\cA(\cO)$.
An element of $\cS(\cO)$ is called a \textit{cusp form}.

Let $\cA_N(\cO)$ be the subspace of right $N_\cO$-invariant elements in $\cA(\cO)$,  which are often regarded as functions on $G_\Q\bs G_\A/N_\cO$.
Set $t_\cO=|G_\Q\bs G_\A/N_\cO|$ and fix  a set of representatives $\{y_j\}_{j=1}^{t_\cO}$ of cosets in $G_\Q\bs G_\A/N_\cO$.
We take $y_j$'s from $G_{\A_f}$.
Note that $\dim_\C\cA_N(\cO)=t_\cO$.
Set $\cS_N(\cO)=\cA_N(\cO)\cap\cS(\cO)$.

If an order $\cO'$ in $D$ contains $\cO$,  we may regard $\cS(\cO')\subset\cS(\cO)$.
Let $\cS^\new(\cO)$ be the orthogonal complement of $\sum_{\cO'}\cS(\cO')$ in $\cS(\cO)$,  where $\cO\subsetneqq\cO'$ are orders in $D$.
Set $\cS_N^\new(\cO)=\cS_N(\cO)\cap\cS^\new(\cO)$.

%

For a prime $p\not\in S_\cO$,  let $\gamma_p$ denote the element of $(D_{\A_f})^\times$ which is the identity at all finite places $v$ other than $p$,  and corresponds to   
    $\renewcommand{\arraystretch}{0.8}
    \begin{pmatrix}
    p & 0 \\
    0 & 1
    \end{pmatrix}$
at $p$ under the isomorphism $\cO_p\simeq\M_2(\Z_p)$.
Take a set of representatives $\{\alpha_j\}$ of $K\gamma_pK/K$ and define the \textit{$p$-th Hecke operator} $T_p$ on $\cA(\cO)$ and $\cA_N(\cO)$ by
    \[
    (T_p\phi)(x)=\sum_j \phi(x\alpha_j),  \qquad \phi\in\cA(\cO).
    \]
The Hecke operators $\{T_p\}_{p\not\in S_\cO}$ are simultaneously diagonalizable. 
A simultaneous eigenvector is called a \textit{Hecke eigenform}.
For a Hecke eigenform,  the number field generated by its all Hecke eigenvalues is called the \textit{Hecke field}.
We say that a Hecke eigenform is \textit{normalized} if it takes values in the ring of integers of its Hecke field.
Any Hecke eigenforms are constant multiple of normalized Hecke eigenforms.
Let $F_\cO$ (resp.\,$F_{N,  \cO}$) be the composite of Hecke fields for all Hecke eigenforms in $\cA(\cO)$ (resp.\,$\cA_N(\cO)$).
This is a totally real finite Galois extension of $\Q$.

We translate the above ad\`elic description of algebraic modular forms into the classical framework using the terminology of quaternion ideals.
A \textit{right fractional $\cO$-ideal} is a $\Z$-lattice $I\subset D$ which verifies $I\alpha\subset I$ for any $\alpha\in\cO$.
Two right fractional $\cO$-ideals $I$ and $J$ are \textit{in the same right class} if there exists $\alpha\in D^\times$ such that $\alpha I=J$ and denote by $[I]$ the right class of $I$.

Let $\Cl(\cO)$ be the set of right classes of right fractional $\cO$-ideals $I$ satisfying $I\otimes_\Z\Z_v=x_v\cO_v$ for each finite place $v$ with some $x_I=(x_v)_v\in D_{\A_f}^\times$.
The map $I\mapsto x_I$ descends to a well-defined bijection from $\Cl(\cO)$ to $G_\Q\bs G_\A/K$.
The order $h_\cO$ of the set $\Cl(\cO)$ is called the \textit{class number} of $\cO$.
When we write $\Cl(\cO)=\{[I_1],  \ldots,  [I_{h_\cO}]\}$,  we always assume that $I_1=\cO$ is the trivial fractional $\cO$-ideal.
The set of representatives of cosets in $G_\Q\bs G_\A/K$ we fixed above will be denoted as $\{x_{[I]}\}_{[I]\in\Cl(\cO)}$ so that each $x_{[I]}$ corresponds to $[I]$.

Let $\Typ(\cO)$ be the set of isomorphism classes of orders $\cO'$ in $D$ \textit{locally isomorphic to} $\cO$,  \textit{i.e.} $\cO'_v=y_v\cO_vy_v^{-1}$ for each finite place $v$ with some $y_{\cO'}=(y_v)\in D_{\A_f}^\times$.
Then the map $\cO'\mapsto y_{\cO'}$ descends to a well-defined bijection from $\Typ(\cO)$ to $G_\Q\bs G_\A/N$.
The order $t_\cO$ of the set $\Typ(\cO)$ is called the \textit{type number} of $\cO$.
When we write $\Typ(\cO)=\{[\cO_1],  \ldots,  [\cO_{t_\cO}]\}$,  we always assume that $\cO_1=\cO$.
The set of representatives of cosets in $G_\Q\bs G_\A/N$ we fixed above will be denoted as $\{y_{[\cO']}\}_{[\cO']\in\Typ(\cO)}$ so that each $y_{[\cO']}$ corresponds to $[\cO']$.

The projection from $G_\Q\bs G_\A/K$ to $G_\Q\bs G_\A/N$ induces a surjective map from $\Cl(\cO)$ to $\Typ(\cO)$.
    \[
    \xymatrix{
    G_\Q\bs G_\A/K \ar @{->>}[r] \ar_{\rotatebox{90}{$\sim$}}[d] 
    \ar@{}[dr]|\circlearrowleft
    & G_\Q\bs G_\A/N \ar^{\rotatebox{90}{$\sim$}}[d] \\
    \Cl(\cO) \ar @{->>}[r]& \Typ(\cO)
    }
    \]
For a right fractional $\cO$-ideal $I\subset D$,  set $\cO(I)=\{\alpha\in D \mid \alpha I\subset I\}$.
It is easy to check that $\cO(I)$ is an order in $D$ which is locally isomorphic to $\cO$,  the isomorphism class of $\cO(I)$ depends only on the right class of $I$ and the map $\Cl(\cO) \twoheadrightarrow \Typ(\cO)$ defined above is given by $[I] \mapsto [\cO(I)]$.
We may regard $\cA(\cO)$ as the space of functions on $\Cl(\cO)$ and $\cA_N(\cO)$ its subspace of functions which factor through $\Typ(\cO)$ or the space of functions on $\Typ(\cO)$.

\subsection{Toric periods}\label{subsec:Toric periods}

We summarize some facts on toric periods.
For a moment,  let $\pi=\otimes_v\pi_v$ be a general irreducible cuspidal automorphic representation of $G_\A$,  which is not a character.
We may remove the assumption that $D$ is definite.
Let $\pi'=\otimes_v\pi'_v$ be the Jacquet-Langlands transfer of $\pi$ to $\PGL_2(\A)$.
Take a cusp form $\phi\in\pi$ with decomposition $\phi=\otimes_v\phi_v$.
For $E\in X(D)$,  let $T=T_E$ be the subtorus of $G$ such that $T(R)=(E\otimes_\Q R)^\times/R^\times$ for a $\Q$-algebra $R$.
Here we fixed an embedding $\iota_0\,\colon E\hookrightarrow D$.
Let $\d t$ be the Tamagawa measure on $T_\A$ and $\d t_v$ the local Tamagawa measure on $T_v$ which satisfies $\d t=L(1,  \eta_E)^{-1}\prod_v \d t_v$.
The \textit{toric period} of $\phi$ with respect to $E$ is the integral
     \[
     \cP_{\iota_0,  E}(\phi)=\int_{T_\Q\bs T_\A}\phi(t) \d t. 
     \]   
Note that the property that $\cP_{\iota_0,  E}\not\equiv0$ on $\pi$ is independent of $\iota_0$.
Hence we say $\cP_E\not\equiv0$ on $\pi$.
    
We fix a $G_v$-invariant inner product $\langle \ ,  \ \rangle_v$ on $\pi_v$ so that $\langle \ ,  \ \rangle=\prod_v\langle \ ,  \ \rangle_v$.
Recall that $\langle \ ,  \ \rangle$ is the Petersson inner product with respect to 2 times the Tamagawa measure on $G_\A$.
We define the local toric period $\alpha_{\iota_0,  E_v}(\phi_v)$ and its normalization $\alpha_{\iota_0,  E_v}^\natural(\phi_v)$ as
    \[
    \alpha_{\iota_0,  E_v}(\phi_v)
    =\int_{T_v}\langle \pi_v(t_v)\phi_v,  \phi_v\rangle_v \d t_v,   \quad
    \alpha_{\iota_0,  E_v}^\natural(\phi_v)=
    \frac{L(1,  \eta_{E,  v})L(1,  \pi'_v,  \Ad)}
    {\zeta_v(2)L(\frac12,  \pi'_v)L(\frac12,  \pi'_v\otimes\eta_{E,  v})}\,
    \alpha_{\iota_0,  E_v}(\phi_v).
    \]
Here,  $L(s,  \pi'_v)$ and $L(s,  \pi'_v\otimes\eta_{E,  v})$ are local standard $L$-factors of $\pi'_v$ and $\pi'_v\otimes\eta_{E,  v}$,  respectively and $L(s,  \pi'_v,  \Ad)$ is the local adjoint $L$-factor of $\pi'_v$.
For $\cE_v\in X(D_v)$ (with fixed embedding $\cE_v\hookrightarrow D_v$),  we similarly define $\alpha_{\cE_v}(\phi_v)$ and $\alpha_{\cE_v}^\natural(\phi_v)$.
Let $\varepsilon(\pi'_v; \cE_v)$ (resp.\,$\varepsilon(\pi'; E)$) be the root number of the base change of $\pi'_v$ to $\PGL_2(\cE_v)$ (resp.\,$\pi'$ to $\PGL_2(\A_E)$).

Waldspurger \cite{Wal2} proved the relation between the toric periods and the central $L$-values.
\begin{theorem}\label{thm:Waldspurger}
We keep the notation.
\begin{itemize}
\item[(1)] If $\cP_E\not\equiv0$ on $\pi$,  then the ramification set of $D$ coincides with the set of places $v$ at which we have $\varepsilon(\pi'_v; E_v)=-1$.
Conversely,  if this condition is satisfied,  then $\cP_E\not\equiv0$ on $\pi$ if and only if $L(\tfrac12,  \pi')L(\tfrac12,  \pi'\otimes\eta_E)\neq0$. 
\item[(2)]For $\phi=\otimes_v\phi_v\in\pi=\otimes_v\pi_v$ as above,  
    \[
    |\cP_{\iota_0,  E}(\phi)|^2
    =\frac{\zeta(2)L(\frac12,  \pi')L(\frac12,  \pi'\otimes\eta_E)}
    {L(1,  \eta_E)^2L(1,  \pi',  \Ad)}\prod_v\alpha_{\iota_0,  E_v}^\natural(\phi_v).
    \]
\end{itemize}
\end{theorem}

Now we return to an algebraic modular form $\phi=\otimes_v\phi_v\in\cA(\cO)$,  in particular $D$ is definite.
Later we will use the following bound for the Euler products in the Waldspurger's formula.

\begin{lemma}\label{lem:bound_period}
There is a positive constant $C$ such that $|\Delta_E|^{\tfrac12}\left|
    \prod_v\alpha_{\iota_0,  E_v}^\natural(\phi_v)\right| \leq C$
for any $E\in X(D)$.
\end{lemma}

\begin{proof}
We take a sufficiently large finite set $S$ of places of $\Q$ so that 
\begin{itemize}
\item $S$ contains $S_\cO\cup\{2\}$;
\item $\phi_v$ is a spherical vector normalized so that $\langle \phi_v,  \phi_v \rangle_v=1$ for any $v\not\in S$;
\item $G_\A=G_\Q\left(\prod_{\substack{v\in S \\ v<\infty}}G_v\right)K$.
\end{itemize}
For each $E\in X(D)$,  we take $\delta_E\in D^\times$ with $\delta_E^2=\Delta_E$ so that it is $K_v$-conjugate to 
    $\overline{\begin{psmallmatrix}
    0 & \Delta_E \\
    1 & 0
    \end{psmallmatrix}}$ in $\PGL_2(F_v)\simeq G_v$
for all $v\not\in S$.
Then $E$ is identified with the subalgebra $\Q(\delta_E)$ of $D$.
Under this identification,  one can see that $\alpha^\natural_{\iota_0,  E_v}(\phi_v)=1$ for  all $v\not\in S$.
A simple calculation shows that $|\Delta_E|^{\frac12}\alpha^\natural_{\iota_0,  E_\infty}(\phi_\infty)$ is independent of $E$ since $\pi_\infty$ is the trivial representation.

It remains to show that $\prod_{\substack{v\in S \\ v<\infty}}\alpha^\natural_{\iota_0,  E_v}(\phi_v)$ is bounded.
For $v\in S\setminus\{\infty\}$ and $\cE_v\in X(D_v)$,  fix $\delta_{\cE_v}\in D^\times_v$ so that $\delta_{\cE_v}^2\in F^\times_v$ and $\cE_v\simeq\Q_v(\delta_{\cE_v})$.
When $E_v\simeq \cE_v$,   there exists $g_{E,  v}\in G_v$ such that $g_{E,  v}^{-1}\delta_Eg_{E,  v}=\delta_{\cE_v}$.
We may assume the representatives $x_1,  \ldots,  x_{h_\cO}$ of cosets in $G_\Q\bs G_\A/K$ are contained in $\prod_{\substack{v\in S \\ v<\infty}}G_v$.
Set $g_E=(g_{E,  v})_{v\in S\setminus\{\infty\}}\in\prod_{\substack{v\in S \\ v<\infty}}G_v$ and take $\gamma_E\in G_\Q$ so that $\gamma_E g_E\in\coprod_{i=1}^{h_\cO}x_iK$.
In particular $\gamma_E\in K_v$ for all $v\not\in S$. 
Hence we can replace $\delta_E$ with $\delta'_E:=\gamma_E\delta_E\gamma_E^{-1}$ and $g_E$ with $g'_E=(g'_{E,  v})_v:=\gamma_E g_E$ to obtain $g'^{-1}_{E,  v}\delta'_Eg'_{E,  v}=\delta_{\cE_v}$ for $v\in S$.
Since $g'_E$ is contained in a compact set $\coprod_{i=1}^{h_\cO}x_iK$,  we obtain $\prod_{\substack{v\in S \\ v<\infty}}\alpha^\natural_{\iota_0,  E_v}(\phi_v)\asymp\prod_{\substack{v\in S \\ v<\infty}}\alpha^\natural_{\cE_v}(\phi_v)$.
This completes the proof.
\end{proof}

We rewrite the toric period $\cP_{\iota_0,  E}(\phi)$ in terms of the counting measure.
For $m\in\Z$ with $m \equiv 0,  1 \pmod 4$,  set $\fo(m)=\Z[\frac{m+\sqrt{m}}{2}]$.
Then,  $\fo=\fo(m)$ is an order in $\Q(\sqrt{m})$ of discriminant $m$. 
An order in $E\in X$ is uniquely written as $\fo(k^2\Delta_E)$ with some $k\in\Z_{>0}$.
For a prime $p$,  set $\fo_p=\fo\otimes_\Z\Z_p$.

Let $\fo$ be an order in $E\in X(D)$ and $\Emb(\fo,  \cO)$ the set of $\iota\in\Emb(E,  D)$ satisfying $\iota(E)\cap \cO=\iota(\fo)$. 
Let $\Emb(\fo,  \cO)_{/\sim}$ denote the set of $\cO^\times$-conjugacy classes in $\Emb(\fo,  \cO)$.
For $\iota\in\Emb(\fo,  \cO)$,  we write its class in $\Emb(\fo,  \cO)_{/\sim}$ by $[\iota]$.
We similarly define $\Emb(\fo_v,  \cO_v)$,   $\Emb(\fo_v,  \cO_v)_{/\sim}$ and $[\iota_v]$ for each $v<\infty$ and $\iota_v\in\Emb(\fo_v,  \cO_v)$.
Suppose $\Emb(\fo,  \cO)\neq\emptyset$ and fix $\iota_0\in\Emb(\fo,  \cO)$.

Let $T=T_E$ be the subtorus of $G$ attached to $\iota_0(E)$.
Set $U_v=\overline{(\fo_v^\times)}\subset T_v$ for $v<\infty$ and $U_\infty=T_\infty$.
Then $U=U_\fo=\prod_v U_v$ is an open compact subgroup of $T_\A$.
Normalize the Haar measure on $T_\A$ by $\vol(U)=1$.
Set $h_\fo=|T_\Q\bs T_\A/U_\fo|$ and $u_\fo=|T_\Q\cap U_\fo|=\#\overline{(\fo^\times)}$.
Then $u_{\fo(-3)}=3$,  $u_{\fo(-4)}=2$ and $u_\fo=1$ otherwise.
Let $\Cl(\fo)$ denote the set of classes of fractional $\fo$-ideals $\fa\subset E$ which satisfies $\fa\otimes_\Z\Z_v=z_v\fo_v$ for each finite place $v$ with some $z_\fa=(z_v)_v\in\A_{E,  f}^\times$.
Then the map $[\fa]\mapsto z_\fa$ defines a well-defined bijection from $\Cl(\fo)$ to $T_\Q\bs T_\A/U$.
It induces a group structure on $\Cl(\fo)$ from that of $T_\Q\bs T_\A/U$.
The order $h_\fo$ of $\Cl(\fo)$ is called the class number of $\fo$.
If $\fo$ is the ring of integers of $E$ (i.e. $m=\Delta_E$),  write $\fo$,  $u_\fo$ and $h_\fo$ as $\fo_E$,  $u_E$ and $h_E$,  respectively.
We say that $E\in X(D)$ has an \textit{optimal embedding} with respect to $\cO$ if $\Emb(\fo_E,  \cO)\neq\emptyset$.

An embedding $\iota_0\in\Emb(\fo,  \cO)$ induces a map from $T_\Q\bs T_\A/U$ to $G_\Q\bs G_\A/K$. 
Hence we obtain a map $i_0\,\colon\Cl(\fo)\rightarrow\Cl(\cO)$ which makes the following diagram commutative: 
    \[
    \xymatrix{
    T_\Q\bs T_\A/U \ar[r] \ar_{\rotatebox{90}{$\sim$}}[d] 
    \ar@{}[dr]|\circlearrowleft
    & G_\Q\bs G_\A/K \ar^{\rotatebox{90}{$\sim$}}[d] \\
    \Cl(\fo) \ar_{i_0}[r]& \Cl(\cO).
    }
    \]
Since $\vol(T_\Q\bs T_\A)=\frac{h_\fo}{u_\fo}$ and the Tamagawa number of $T$ is $2$,  we have $\cP_{\iota_0,  E}(\phi)=\frac{2u_\fo}{h_\fo}\fP_{\iota_0, \fo}(\phi)$,  where 
\begin{equation}\label{eq:period}
    \fP_{\iota_0,  \fo}(\phi)
    :=\frac{1}{u_\fo}\sum_{[\fa]\in\Cl(\fo)}\phi(i_0([\fa]))
    =\frac{1}{u_\fo}\sum_{[I]\in\Cl(\cO)}\#\Big(i_0^{-1}([I])\Big)\cdot\phi([I]).
\end{equation}
We also call $\fP_{\iota_0,  \fo}(\phi)$ a toric period of $\phi$.
When $\fo=\fo_E$,  we often write $\fP_{\iota_0,  E}(\phi):=\fP_{\iota_0,  \fo_E}(\phi)$.
If $\phi$ is in $\cA_N(\cO)$,  then \eqref{eq:period} becomes
    \[
    \fP_{\iota_0,  \fo}(\phi)
    =\frac{1}{u_\fo}\sum_{[\cO']\in\Typ(\cO)}
    \#\Big(i_0^{-1}([\cO'])\Big)\cdot\phi([\cO']). 
    \]
Here,  by abuse of notation,  $i_0$ denotes the map $\Cl(\fo)\rightarrow\Typ(\cO)$ induced from $\iota_0$.


\section{Classical Waldspurger's lift}
\label{sec:Classical}

In this section,  we review the half-integral weight modular form $\cW(\phi)$ associated with $\phi$.
Following \cite{BSP1},  we call it the \textit{classical Waldspurger's lift} of $\phi$.
See also \cite[\S12--13]{Gross}.
The Fourier coefficients of $\cW(\phi)$ are closely related to toric periods $\fP_{\iota_0,  E}(\phi)$.

\subsection{Definition}
\label{subsec:Definition}

For a right fractional $\cO$-ideal $I$,  we put $L_I=\{v\in \Z+2\cO(I) \mid \Tr(v)=0\}$,  where $\Tr$ is the reduced trace on $D$.
Let $\theta_\cO$ be the ternary theta function on $\Cl(\cO)\times \fH$ given by
    \[
    \theta_\cO([I],  z)=\sum_{v\in L_I}q^{\Nm(v)},  \qquad q=e^{2\pi\sqrt{-1}z}, 
    \]
where $\fH$ is the upper half plane and $\Nm$ is the reduced norm on $D$.
The right hand depends only on $[I]$ and this is well-defined.
Moreover,  the right hand side depends only on $[\cO(I)]$,  thus $\theta_\cO$ factors through the natural surjection $\Cl(\cO)\times\fH \twoheadrightarrow \Typ(\cO)\times\fH$.
The classical Wldspurger's lift $\cW(\phi)$ of $\phi$ is a modular form of weight $\frac32$ given by the inner product
    \[
    \cW(\phi,  z)=(\phi,  \overline{\theta_\cO(- ,  z )}),  \qquad z\in\fH.
    \] 

The image of the map $\cW$ is studied in \cite[p.379]{BSP1} and \cite[Proposition 12.9]{Gross}.
Let $M$ be a square-free positive odd integer.
We detnote by $S^+_{3/2}(M)^\new$  the space of cuspidal newforms $g(z)=\sum_{n=1}^\infty a(n)q^n$ of weight $\frac32$ on $\Gamma_0(4M)$ with $a(n)=0$ unless $n\equiv 0,3 \pmod 4$.
Let $S_2^\new(M)$ be the space of cuspidal newforms of weight $2$ on $\Gamma_0(M)$.
For a Hecke eigenform $f\in S_2^\new(M)$,  let $S_{3/2}^+(M,  f)^\new$ be the subspace of $S_{3/2}^+(M)^\new$ of cusp forms whose Hecke eigenvalue at $p^2$ equals that of $f$ at $p$ for almost all primes $p\nmid 2M$.
Let $L(s,  f)$ denote the standard $L$-function of $f$.
According to \cite[Proposition 1]{Wal1} and \cite[Theorem 2]{Kohnen1},  we have the Shimura decomposition
    \[
    S^+_{3/2}(M)^\new=\bigoplus_{f\in S_2^\new(M)} S_{3/2}^+(M,  f)^\new,
    \]
where $f$ runs through Hecke eigenforms in $S_2^\new(M)$.

\begin{theorem}
Keep the above notation.

\begin{itemize}
\item[(1)] Suppose $\disc(\cO)$ is odd and square-free.
Then $\cW(\phi)$ belongs to $S_{3/2}^+(\disc(\cO))$.
\item[(2)] Let $M$ be a square-free positive odd integer.
The map $\phi\mapsto\cW(\phi)$ induces a surjective map
    \[
    \bigoplus_{\substack{(D,  \cO) \\ \disc(\cO)=M}}\cS^\new(\cO) 
    \ \twoheadrightarrow
    \bigoplus_{\substack{f\in S_2^\new(M) \\ L(1,  f)\neq0}} 
    S_{3/2}^+(M,  f)^\new.
    \]
Here,  $(D,  \cO)$ runs through pairs of a definite quaternion algebras $D$ over $\Q$ and Eichler orders $\cO$ in $D$ with $\disc(\cO)=M$ and $f$ runs through Hecke eigenforms in $S_2^\new(M)$ such that $L(1,  f)\neq0$. 
The kernel of this map is $\bigoplus_{\substack{(D,  \cO) \\ \disc(\cO)=M}}\cS_N^\new(\cO)^\perp$,  where $\cS_N^\new(\cO)^\perp$ is the orthogonal complement of $\cS_N^\new(\cO)$ in $\cS^\new(\cO)$.
\end{itemize}
\end{theorem}

The argument of \cite[Proposition 12.9]{Gross} shows that the Fourier expansion of $\cW(\phi)$ becomes 
\begin{equation}\label{eq:Fourier}
    \cW(\phi,  z)=\sum_{\substack{n\in\Z_{>0} \\ n\equiv 0,  3 
    \hspace{-.6em}\pmod 4}}\left(\sum_{[I]\in\Cl(\cO)}
    \sum_{\fo(-n)\subset\fo} u_\fo^{-1}\#\Big(\Emb(\fo,  \cO(I))_{/\sim}\Big) 
     \phi([I])\right) q^n.
\end{equation}
Here,  the inner-most sum is over orders in $\Q(\sqrt{-n})$ which contain $\fo(-n)$.
Let $a_\phi(n)$ denote the $n$-th Fourier coefficient of $\cW(\phi)$.
For an imaginary quadratic field $E\in X$,  we have
\begin{equation}\label{eq:Fourier-2}
    a_\phi(|\Delta_E|)=\frac{1}{u_E}\sum_{[I]\in\Cl(\cO)}
     \#\Big(\Emb(\fo_E,  \cO(I))_{/\sim}\Big)\phi([I]).
\end{equation}
The goal of this section is to relate $a_\phi(|\Delta_E|)$ to the toric period $\fP_{\iota_0,  E}(\phi)$ for $E\in X(D)$.

\subsection{Action of ideal class groups}
\label{subsec:Action}

Following \cite[Theorem 30.4.7 and Corollary 30.4.23]{Voight},  we introduce a group action of $\Cl(\fo)$ on $\coprod_{[I]\in\Cl(\cO)}\Emb(\fo,  \cO(I))_{/\sim}$.
It provides a useful description of the set $i_0^{-1}([I])$ which appears in \eqref{eq:period}.

First we need to clarify the notation $\Emb(\fo,  \cO(I))_{/\sim}$ for a right class $[I]$.
For right fractional $\cO$-ideals $I,  J\subset D$ satisfying $J=\alpha I$ with $\alpha\in D^\times$,  the conjugation by $\alpha$ defines a bijection from $\Emb(\fo,  \cO(I))$ to $\Emb(\fo,  \cO(J))$.
Then we identify $\Emb(\fo,  \cO(I))_{/\sim}$ with $\Emb(\fo,  \cO(J))_{/\sim}$ under the bijection induced from it.
Note that this identification does not depend on the choice of $\alpha$.
Hence one can speak of $\Emb(\fo,  \cO(I))_{/\sim}$ for $[I]\in\Cl(\cO)$\footnote{Another way to define $\coprod_{[I]\in\Cl(\cO)}\Emb(\fo,  \cO(I))_{/\sim}$ is to consider the $D^\times$-conjugate action on the set $\coprod_I \Emb(\fo,  \cO(I))_{/\sim}$,  where $I$ runs over all right fractional $\cO$-ideals.}.

Assume that $\Emb(\fo,  \cO)\neq\emptyset$ and fix $\iota_0\in\Emb(\fo,  \cO)$.
Set $B(\cO)_\Q=\{g\in G_\Q \mid g^{-1}\cdot\iota_0\in\Emb(\fo,  \cO)\}$.
Here,  $\cdot$ denotes the conjugate action.
Similarly,  set $B_p(\cO_p)=\{g\in G_p \mid g^{-1}\cdot(\iota_{0}\otimes id_{\Q_p})\in\Emb(\fo_p,  \cO_p)\}$ for each prime $p$ and $B(\cO)_\A=\prod_p B_p(\cO_p)\times G_\R$,  where $p$ runs over all primes. 
For $[I]\in\Cl(\cO)$,  we write the corresponding coset as $G_\Q x_{[I]}K$.
Then one sees that $(B(\cO)_\A\cap G_\Q x_{[I]}K)/K \simeq B(\cO(I))_\Q/\overline{\cO(I)^\times}$.
Here,  $B(\cO(I))_\Q/\overline{\cO(I)^\times}$ and $B(\cO(J))_\Q/\overline{\cO(J)^\times}$ are identified with each other by right multiplication with $\alpha\in D^\times$ if $J=\alpha I$.
Hence we get a bijection from $T_\Q\bs B(\cO)_\A/K$ to $\coprod_{[I]\in\Cl(\cO)}T_\Q\bs B(\cO(I))_\Q/\overline{\cO(I)^\times}$.
Since the map sending $g\in B(\cO(I))_\Q$ to 
    \[
    \xymatrix{T_\Q\bs T_\A/U \times T_\A\bs B(\cO)_\A/K \ar^(0.6){\sim}[r]
     \ar_{\rotatebox{90}{$\sim$}}[d]
    &T_\Q\bs B(\cO)_\A/K \ar^(0.35){\sim}[r]
    &\ds\coprod_{[I]\in\Cl(\cO)}T_\Q\bs B(\cO(I))_\Q/\overline{\cO(I)^\times}
      \ar^(0.55){\rotatebox{90}{$\sim$}}[d] \\
    \Cl(\fo)\times\ds\prod_p\Emb(\fo_p,  \cO_p)_{/\sim} 
    &
    & \ds\coprod_{[I]\in\Cl(\cO)}\Emb(\fo,  \cO(I))_{/\sim}.
    }
    \]
$g^{-1}\cdot\iota_0$ descends to a bijection from $T_\Q\bs B(\cO(I))_\Q/\overline{\cO(I)^\times}$ to $\Emb(\fo,  \cO(I))_{/\sim}$,  the left multiplication of $T_\Q\bs T_\A/U$ on $T_\Q\bs B(\cO)_\A/K$ defines an action of $\Cl(\fo)$ on $\coprod_{[I]\in\Cl(\cO)}\Emb(\fo,  \cO(I))_{/\sim}$.
This action is fixed-point free and hence each orbit has $h_\fo$ elements. 
We write this action as $[\fa]\ast[\iota]$ for $[\fa]\in\Cl(\fo)$ and $[\iota]\in\Emb(\fo,  \cO(I))_{/\sim}$.
The leftmost horizontal arrow in the above diagram is a bijection induced from this action.
Note that it is not canonical and depends on the choice of orbit representatives. 

Let $[\iota_0]$ denote the class in $\Emb(\fo,  \cO(I_1))_{/\sim}$ associated with $\iota_0$,  where $[I_1]\in\Cl(\cO)$ is the trivial class.
It is easy to check that 
\begin{equation}\label{eq:action}
    i_0^{-1}([I])=\{[\fa]\in\Cl(\fo) \mid [\fa]\ast[\iota_0]\in\Emb(\fo,  \cO(I))_{/\sim}\}
\end{equation}
for $[I]\in\Cl(\cO)$.
In particular,  $\Emb(\fo,  \cO(I))\neq\emptyset$ if $i_0^{-1}([I])\neq\emptyset$.
One sees that the toric period $\fP_{\iota_0,  \fo}(\phi)$ depends only on the $\Cl(\fo)$-orbit of $[\iota_0]$.

Note that the map sending $g\in B_p(\cO_p)$ to $g^{-1}\cdot(\iota_0\otimes id_{\Q_p})$ induces a bijection from $T_p\bs B_p(\cO_p)/K_p$ to $\Emb(\fo_p, \cO_p)_{/\sim}$.
Hence $[\iota]$,  $[\iota']\in\coprod_{[I]\in\Cl(\cO)}\Emb(\fo,  \cO(I))_{/\sim}$ are in the same $\Cl(\fo)$-orbit if and only if $[\iota\otimes id_{\Q_p}]=[\iota'\otimes id_{\Q_p}]$ in $\Emb(\fo_p,  \cO_p)_{/\sim}$ for all primes $p$.
In particular,  the number of $\Cl(\fo)$-orbits is $|T_\A\bs B(\cO)_\A/K|=\prod_p |T_p\bs B_p(\cO_p)/K_p|=\prod_p\#\Big(\Emb(\fo_p, \cO_p)_{/\sim}\Big)$.

The local embedding number $\#\Big(\Emb(\fo_p, \cO_p)_{/\sim}\Big)$ is closely studied in \cite[\S\,30.5 and \S\,30.6]{Voight}.
We need the following result for the case where $\fo=\fo_E$ is the ring of integers of $E\in X(D)$.

\begin{lemma}\label{lem:local_emb}
For $E\in X(D)$ and its ring of integers $\fo_E$,  
    \[
    \#\Big(\Emb(\fo_{E,  p},  \cO_p)_{/\sim}\Big)=
        \begin{cases}
        1 & \text{if $p \nmid \disc(\cO)=\disc(D)\level(\cO)$,} \\
        1-\left(\frac{\Delta_E}{p}\right) & \text{if $p \mid \disc(D)$,}\\
        1+\left(\frac{\Delta_E}{p}\right) & \text{if $p \mid \level(\cO)$.}
        \end{cases}
    \]
When $\Emb(\fo_{E,  p},  \cO_p)_{/\sim}\simeq T_p\bs B_p(\cO_p)/K_p$ has two elements,  the non-trivial one is given as follows:
\begin{itemize}
\item[(1)] If $p \mid \disc(D)$ and $p$ is inert and unramified in $E$,  then the non-trivial coset in $T_p\bs B(\cO_p)/K_p$ is represented by a uniformizing element of $D_p$,  \textit{i.e.}\,an element $\varpi_p\in D_p^\times$ such that $\Nm_p(\varpi_p)\in p\Z_p^\times$,  where $\Nm_p$ is the reduced norm on $D_p$.
\item[(2)] If $p \mid \level(\cO)$ and $p$ splits in $E$,  then the non-trivial coset in $T_p\bs B(\cO_p)/K_p$ is  
    $T_p\begin{psmallmatrix}
    0 & 1 \\
    p & 0
    \end{psmallmatrix}K_p$.
\end{itemize}
In particular,  
    \[
    |T_\A\bs B(\cO)_\A/K|=
    \prod_{p \mid \disc(D)}\left(1-\left(\frac{\Delta_E}{p}\right)\right)
   \prod_{p \mid \level(\cO)}\left(1+\left(\frac{\Delta_E}{p}\right)\right),  
    \]
which is a power of $2$ and $|T_p\bs B_p(\cO_p)/N_p|=|T_\A\bs B(\cO)_\A/N|=1$.
\end{lemma}

\begin{remark}\label{rem:dependence}
Suppose $\iota_1,  \iota_2\in\Emb(\fo_E,  \cO)$ and let $i_j\,\colon \Cl(\fo_E)\rightarrow \Typ(\cO)$ be the induced map. 
Then we have $i_1=i_2$ by \cref{lem:local_emb}.
In particular,  $\fP_{\iota_1,  E}(\phi)=\fP_{\iota_2,  E}(\phi)$ for $\phi\in\cA_N(\cO)$.
\end{remark}

\subsection{Fourier coefficients}
\label{subsec:Fourier}

From now till the end of this section,  suppose that $\phi\in\cS^\new(\cO)$ is a Hecke eigenform.
Let $\pi=\otimes_v \pi_v$ be the irreducible cuspidal automorphic representation of $G_\A$ generated by $\phi$.
The \textit{Atkin-Lehner sign} $\AL(\pi_p)$ of $\pi_p$ for a prime $p$ is given as
    \[
    \AL(\pi_p)=
        \begin{cases}
        \varepsilon(\pi_p) & \text{if $p \nmid \disc(D)$,} \\
        \chi_p(p) &\text{if $p \mid \disc(D)$ and $\pi_p=\chi_p\circ\Nm_p$,}
        \end{cases}
    \]
where $\varepsilon(\pi_p)$ is the root number of $\pi_p$,  $\chi_p$ is an unramified character of $\Q_p^\times$ and $\Nm_p$ is the reduced norm on $D_p$.
Note that $\pi_p$ is a representation of $\PGL_2(\Q_p)$ if $p \nmid \disc(D)$.
We fix $E\in X(D)$ and $\iota_0\in\Emb(\fo_E,  \cO)$,  assuming that $E$ has an optimal embedding with respect to $\cO$.
Let $[\iota_0]\in\Emb(\fo_E,  \cO(I_1))_{/\sim}$ be the class of $\iota_0$,  where $[I_1]\in\Cl(\cO)$ is the trivial class.

\begin{definition}
\begin{itemize}
\item[(1)] For $[\iota],  [\iota']\in \coprod_{[I]\in\Cl(\cO)}\Emb(\fo_E,  \cO(I))_{/\sim}$,  we set
    \[
    \AL_\pi([\iota],  [\iota'])=\prod_{p}\AL_{\pi_p}([\iota],  [\iota']).
    \]
Here,  $p$ runs over all primes and 
    \[
    \AL_{\pi_p}([\iota],  [\iota']):=
        \begin{cases}
        \AL(\pi_p) & \text{if $[\iota\otimes id_{\Q_p}] \neq 
        [\iota'\otimes id_{\Q_p}]$,} \\
        1 & \textit{otherwise.}
        \end{cases}
    \]
Note that $\AL_\pi([\iota],  [\iota'])$ depends only on the $\Cl(\fo_E)$-orbits of $[\iota]$ and $[\iota']$.

\item[(2)] We put 
    \[
    c_\pi(E)=\sum_{j=1}^{|T_\A\bs B(\cO)_\A/K|} 
    \AL_\pi([\iota_0],  [\iota_j]), 
    \]
where $\{[\iota_j]\}_j$ is a set of representatives of $\Cl(\fo_E)$-orbits in $\coprod_{[I]\in\Cl(\cO)}\Emb(\fo_E,  \cO(I))_{/\sim}$.
\end{itemize}
\end{definition}

We get an equality between the Fourier coefficient $a_\phi(|\Delta_E|)$ and the toric period $\fP_{\iota_0,  E}(\phi)$.

\begin{theorem}\label{thm:Fourier}
Let $\phi\in\cS^\new(\cO)$ be a Hecke eigenform and $\pi=\otimes_v \pi_v$ the irreducible cuspidal automorphic representation of $G_\A$ generated by $\phi$.
Take $E\in X(D)$ with $\Emb(\fo_E,  \cO)\neq\emptyset$.
\begin{itemize}
\item[(1)] The $|\Delta_E|$-th Fourier coefficient of $\cW(\phi)$ satisfies $a_\phi(|\Delta_E|)=c_\pi(E) \cdot \fP_{\iota_0,  E}(\phi)$.
\item[(2)] Let $S_\pi^{\pm}(E)$ be the set of primes $p \mid \disc(\cO)$ such that $\#\Big(\Emb(\fo_{E,  p},  \cO_p)_{/\sim}\Big)=2$ and $\AL(\pi_p)=\pm1$.
Then
    \[
    c_\pi(E)=
        \begin{cases}
        |T_\A\bs B(\cO)_\A/K| & \text{if $S_\pi^-(E)=\emptyset$, } \\
        0 & \text{otherwise}.
        \end{cases}
    \]
In particular,  if $\phi\in\cS_N^\new(\cO)$,  then $c_\pi(E)\neq0$ for any $E\in X(D)$ with $\Emb(\fo_E,  \cO)\neq\emptyset$.

\end{itemize}
\end{theorem}

\begin{proof}
(1) Let $\{[\iota_j]\}_j$ be a set of representatives of $\Cl(\fo_E)$-orbits in $\coprod_{[I]\in\Cl(\cO)}\Emb(\fo_E,  \cO(I))_{/\sim}$.
For each $[\iota_j]$,  we define $\fP_{\iota_j,  E}$ as follows.
Take $[I_j]\in\Cl(\cO)$ and $g_j\in B(\cO(I_j))_\Q$ so that $[\iota_j]\in\Emb(\fo_E,  \cO(I_j))_{/\sim}$ and $\iota_j=g_j^{-1}\cdot\iota_0$.
Set $\cO_j=g_j\cO(I_j)g_j^{-1}$.
Then $R_G(g_jx_{[I_j]})\phi\in\cS(\cO_j)$.
Since $\iota_j\in\Emb(\fo_E,  \cO(I_j))$,  one sees that $\iota_0\in\Emb(\fo_E,  \cO_j)$.
We define $\fP_{\iota_j,  E}(\phi)$ as $\fP_{\iota_0,  E}(R_G(g_jx_{[I_j]})\phi)$.

Let $i_j\,\colon \Cl(\fo_E)\rightarrow\Cl(\cO)$ be the composition of the map $\Cl(\fo_E)\rightarrow\Cl(\cO_j)$ induced from $\iota_0$ and the bijection $\Cl(\cO_j)\xrightarrow{\sim}\Cl(\cO)$ obtained from the right multiplication by $g_jx_{[I_j]}$.
Then
    \[
    \fP_{\iota_j,  E}(\phi)=\frac{1}{u_E}
    \sum_{[I]\in\Cl(\cO)}\#\Big(i_j^{-1}([I])\Big) \cdot \phi([I])
    \]
and $i_j^{-1}([I])=\{[\fa]\in\Cl(\fo_E) \mid [\fa]\ast[\iota_j]\in\Emb(\fo_E,  \cO(I))_{/\sim}\}$.
Since $\sum_j \#\Big(i_j^{-1}([I])\Big)$ equals $\#\Big(\Emb(\fo_E,  \cO(I))_{/\sim}\Big)$,  we obtain
    \[
    \sum_{j=1}^{|T_\A\bs B(\cO)_\A/K|} \fP_{\iota_j,  E}(\phi)=
    \frac{1}{u_E}\sum_{[I]\in\Cl(\cO)}\#\Big(\Emb(\fo_E,  \cO(I))_{/\sim}\Big)
    \cdot \phi([I]).
    \] 
Since $\fP_{\iota_j,  E}(\phi)$ depends only on the $\Cl(\fo_E)$-orbit of $[\iota_j]$,  the left hand side is well-defined,  \textit{i.e.} it does not depend on the choice of $\iota_j$'s.
From \eqref{eq:Fourier-2},  the right hand side equals $a_\phi(|\Delta_E|)$.
Hence we are reduced to the equality $\fP_{\iota_j,  E}(\phi)=\AL_\pi([\iota_0],  [\iota_j])\cdot \fP_{\iota_0,  E}(\phi)$ for each $j$.

Write $g_jx_{[I_j]}\in G_{\A_f}$ as $y=(y_v)_{v<\infty}$.
Then each $y_p$ is a representative of a coset in $T_p\bs B_p(\cO_p)/K_p$ corresponding to $[\iota_j\otimes id_{\Q_p}]$.
Since $\phi$ is a Hecke eigenform,  it is decomposable as $\phi=\otimes_v\phi_v$ with $\phi_v\in\pi_v$.
For every $v<\infty$,  $\phi_v$ is a local newform.
We have $R_G(g_jx_{[I_j]})\phi=\otimes_v\pi_v(y_v)\phi_v$ as an element of $\pi=\otimes_v\pi_v$.
From \cref{lem:local_emb} and the newform theory (e.g. \cite[Theorem 3.2.2]{Schmidt}),  one sees that $\pi_p(y_p)\phi_p=\AL_{\pi_p}([\iota_0],  [\iota_j])\phi_p$.
Since $\fP_{\iota_0, E}$ on $\cS(\cO)$ and that on $\cS(\cO_j)$ are  restrictions of a common linear form $(2u_E)^{-1}h_E\cP_{\iota_0,  E}$ on $\pi$,  we obtain the desired equality.

(2) One sees that $c_\pi(E)=\sum_{S^+}\sum_{S^-} (-1)^{|S^-|}$,  where $S^\pm$ runs over subsets of $S_\pi^\pm(E)$.
Thus $c_\pi(E)=2^{|S_\pi^+(E)|}$ if $S_\pi^-(E)=\emptyset$ and $c_\pi(E)=0$ otherwise. 
If $S_\pi^-(E)=\emptyset$,  then $2^{|S_\pi^+(E)|}$ equals $\prod_p\#\Big(\Emb(\fo_{E,  p},  \cO_p)\Big)=|T_\A\bs B(\cO)_\A/K|$.
This proves the first assertion.
If $\phi\in\cS_N^\new(\cO)$,  we have $\AL(\pi_p)=1$ for every $p$.
Hence $S_\pi^-(E)=\emptyset$ for any $E\in X(D)$ with $\Emb(\fo_E,  \cO)\neq\emptyset$.
\end{proof}

\begin{remark}
\begin{itemize}
\item[(1)] Since $a_\phi(|\Delta_E|)$ and $c_\pi(E)$ are independent of the choice of $\iota_0$,  we see that $c_\pi(E)=0$ if $\fP_{\iota_0,  E}(\phi) \neq \fP_{\iota',  E}(\phi)$ for some $\iota'\in\Emb(\fo_E,  \cO)$.
\item[(2)] Suppose that $\phi\in\cS^\new(\cO)$ is in the orthogonal complement of $\cS_N^\new(\cO)$.
If $E\in X(D)$ has optimal embedding with respect to $\cO$ and satisfies $S_\pi^-(E)=\emptyset$,  we have $\fP_{\iota,  E}(\phi)=0$ for any $\iota\in\Emb(\fo_E,  \cO)$ since $\cW(\phi)=0$ and $c_\pi(E)\neq0$.
Moreover,  one can deduce $\cP_E\equiv0$ on $\pi$  in that situation.
Since $\phi$ is orthogonal to $\cS_N^\new(\cO)$,  there is at least one prime factor $p$ of $\disc(\cO)$ at which $\pi_p$ is the Steinberg representation or $\pi_p=\chi_p\circ\Nm_p$ with the non-trivial unramified quadratic character $\chi_p$ on $\Q_p^\times$.
On the other hand,  $S_\pi^-(E)=\emptyset$ if and only if $E$ is ramified at all such places. 
Hence we see $\cP_E\equiv0$ on $\pi$ from \cref{thm:Waldspurger} (1).
\end{itemize}
\end{remark}

\subsection{Dependence on Eichler order}
\label{subsec:Dependence}

So far we considered the case where $\Emb(\fo,  \cO)\neq\emptyset$. 
Now we assume $\bigcup_{[\cO']\in\Typ(\cO)}\Emb(\fo,  \cO')\neq\emptyset$ in stead of $\Emb(\fo,  \cO)\neq\emptyset$.
According to \cref{lem:local_emb} and the argument before it,  this is equivalent to that $E$ splits or ramifies at all prime factors of $\level(\cO)$.

Suppose $\Emb(\fo,  \cO')\neq\emptyset$ with $[\cO']\in\Typ(\cO)$.
Take $y=(y_v)_v\in G_{\A_f}$ so that $\cO'_v=y_v\cO_vy_v^{-1}$ for any finite place $v$.
Let $\tau_y\,\colon\Cl(\cO') \xrightarrow{\sim} \Cl(\cO)$ be the map induced from the right multiplication by $y$.
Then $R_G(y)\phi=\phi \circ \tau_y$ is in $\cA(\cO')$,  where $R_G$ is the right translation.
Fix $\iota'_0\in\Emb(\fo,  \cO')$ and let $i'_0$ be the map $\Cl(\fo) \rightarrow \Cl(\cO')$ associated with $\iota'_0$.
Set $i_0:=\tau_y \circ i'_0\,\colon\Cl(\fo)\rightarrow\Cl(\cO)$. 
Then we have $\cP_{\iota'_0,  E}(R_G(y)\phi)=2u_\fo h_\fo^{-1}\fP_{\iota'_0,  \fo}(R_G(y)\phi)$,  where 
    \[
    \fP_{\iota'_0,  \fo}(R_G(y)\phi)
    :=\frac{1}{u_\fo}\sum_{[I]\in\Cl(\cO)}\#\Big(i_0^{-1}([I])\Big)\cdot \phi([I]).
    \]
The map $\phi \mapsto R_G(y)\phi$ defines an isomorphism $\cA_N(\cO)\xrightarrow{\sim} \cA_N(\cO')$,  which does not depend on the choice of $y$.
Thus we write $\fP_{\iota'_0,  \fo}(R_G(y)\phi)$ as $\fP_{\iota'_0,  \fo}(\phi)$ if $\phi$ is in $\cA_N(\cO)$.
If moreover $\fo=\fo_E$,  then $\fP_{\iota'_0,  E}(\phi)$ is independent of $\iota'_0$ by \cref{rem:dependence}.
Hence we write it as $\fP_E(\phi)$.

The ternary theta function $\theta_{\cO'}$ on $\Cl(\cO')\times\fH$ satisfies
    \[
    \theta_{\cO'}([I'],  z)=\theta_\cO(\tau_y([I']),  z),  
    \qquad [I']\in\Cl(\cO'),  \ z\in\fH.
    \]
Therefore the classical Waldspurger's lift of $\phi\in\cA(\cO)$ equals that of $R_G(y)\phi\in\cA(\cO')$.   
Combining these observations,  we can remove the requirement $\Emb(\fo_E,  \cO)\neq\emptyset$ in \cref{thm:Fourier}.

\begin{cor}\label{cor:Fourier}
Let $\phi\in\cS_N^\new(\cO)$ be a Hecke eigenform.
For $E\in X(D)$,  set $c(E):=2^m$,  where $m$ denotes the number of prime factors of $\disc(\cO)$ which is unramified in $E$.
Then $a_\phi(|\Delta_E|)=c(E)\fP_E(\phi)$ if $E$ splits or ramifies at all prime factors of $\level(\cO)$ and $a_\phi(|\Delta_E|)=0$ otherwise.
\end{cor}

\begin{proof}
From \eqref{eq:Fourier-2},  $a_\phi(|\Delta_E|)=0$ unless $\bigcup_{[\cO']\in\Typ(\cO)}\Emb(\fo_E,  \cO')\neq\emptyset$.
This proves the second assertion.
The first assertion follows immediately from \cref{thm:Fourier} and the above argument.
\end{proof}

When $\cO$ is a maximal order,  we see that the non-vanishing of toric periods is equivalent to the non-vanishing of the Fourier coefficient of the Waldspurger's lift.

\begin{prop}\label{prop:Fourier}
Suppose that $\cO$ is a maximal order. 
Let $\phi\in\cS_N^\new(\cO)$ be a Hecke eigenform and $\pi$ the irreducible automorphic representation of $G_\A$ generated by $\phi$.
Then the following three conditions on $E\in X(D)$ are equivalent:
{\rm (1)} $\cP_E\not\equiv 0$ on $\pi$; \quad
{\rm (2)} $\fP_E(\phi)\neq0$; \quad
{\rm (3)} $a_\phi(|\Delta_E|)\neq0$.
\end{prop}

\begin{proof}
The equivalence of (2) and (3) is immediate from \cref{cor:Fourier}.
Take $[\cO']\in\Typ(\cO)$ and $y=(y_v)_v\in G_{\A_f}$ so that $E$ has an optimal embedding with respect to $\cO'$ and $\cO'_v=y_v\cO_vy_v^{-1}$.
Set $\phi'=R_G(y)\phi$.
Then $\phi'$ is a Hecke eigneform in $\cS_N^\new(\cO')$.
For $\iota'\in\Emb(\fo_E,  \cO')$,  we have $(2u_E)^{-1}h_E \cdot \cP_{\iota',  E}(\phi')=\fP_{\iota',  E}(\phi')=\fP_E(\phi)$.
Hence (2) implies (1).

Conversely,  we suppose $\cP_E\not\equiv0$ on $\pi$ and deduce $\fP_E(\phi)\neq0$. 
It suffices to show $\cP_{\iota',  E}(\phi')\neq0$.
Without loss of generality,  we may assume that $E$ has an optimal embedding with respect to $\cO$,  \textit{i.e.}\,$\cO'=\cO$,  $\phi'=\phi$ and $\iota'=\iota\in\Emb(\fo_E,  \cO)$.
Since $\phi$ is a Hecke eigenform,  it is decomposable as $\phi=\otimes_v\phi_v$ with $\phi_v\in\pi_v$.
From \cref{thm:Waldspurger} (2),  the problem is reduced to $\alpha_{\iota,  E_v}(\phi_v)\neq0$ for all $v$.
If $D_v$ is division,  $\pi_v$ is the trivial representation of $G_v$ and $\alpha_{\iota,  E_v}(\phi_v)=\vol(T_v)\langle \phi_v,  \phi_v \rangle_v\neq0$.
Otherwise,  $K_v$ is a maximal compact subgroup of $G_v\simeq\PGL_2(\Q_v)$ and $\phi_v$ is $K_v$-invariant.
The explicit computation in this case is done in \cite[Lemma 2,3]{Wal2} and in particular $\alpha_{\iota,  E_v}(\phi_v)\neq0$.
\end{proof}

\section{Non-vanishing results and Goldfeld's conjecture}
\label{sec:Non-vanishing}

\subsection{Congruence and non-vanishing}
\label{subsec:Congruence}

First we show the following congruence result.
For $r\in\Q^\times$ and a prime $p$,  we write $p\mid r$ if $p$ divides the numerator of $r$.
\begin{lemma}\label{lem:cong}
Let $p$ be an odd prime such that $p\mid \mass(\cO)$.
Then there exists $\varphi\in\cS_N(\cO)$ which takes values in $1+p\Z$.
\end{lemma}

\begin{proof}
This is a special case of \cite[Proposition 2.1]{MW}.
See also \cite[Theorem 2.1]{Martin}.
For completeness we provide a proof.
Suppose there is a desired element $\varphi\in\cS_N(\cO)$ and write its values as $\varphi(y_j)=1+pz_j$ with $z_j\in\Z$,  $j=1,  \ldots,  t_\cO$.
Since $\varphi$ is a cusp form,  $\sum_{j=1}^{t_\cO}k_jw(y_j)^{-1}(1+pz_j)=0$,  where $k_j$ is the order of the fiber at $y_j$ of $\Cl(\cO)\twoheadrightarrow\Typ(\cO)$.
Note that the map $w$ factors through $G_\Q\bs G_\A/N_\cO$. 
Multiplying by $\prod_{j=1}^{t_\cO}w(y_j)$,  we get
    \begin{equation}\label{eq:cong}
    \sum_{j=1}^{t_\cO}k_j\prod_{l\neq j}w(y_l)
    =-p\sum_{j=1}^{t_\cO}k_j\left(\prod_{l\neq j}w(y_l)\right)\cdot z_j.
    \end{equation}
The left hand side equals $\mass(\cO)\prod_{j=1}^{t_\cO}w(y_j)$.    
Conversely,  the desired cusp form exists if there is $z_j\in\Z$,  $j=1,  \ldots,  t_\cO$ which satisfies \eqref{eq:cong}.
Hence it suffices to show that $\gcd\{k_j\prod_{l\neq j}w(y_l)\}_{j=1}^{t_\cO}$ divides $p^{-1}$ times the left hand side of \eqref{eq:cong}. 
    
Let $\gcd\{k_j\prod_{l\neq j}w(y_l)\}_{j=1}^{t_\cO}=\prod q^{e_q}$ be the prime factorization of $\gcd\{k_j\prod_{l\neq j}w(y_l)\}_{j=1}^{t_\cO}$.
For $q\neq p$,  it is obvious that $q^{e_q}$ divides $p^{-1}$ times the left hand side of \eqref{eq:cong}. 
Now it is enough to show that $p^{e_p+1}$ divides $\mass(\cO)\prod_{j=1}^{t_\cO}w(y_j)$. 
Since each fiber of $\Cl(\cO)\twoheadrightarrow\Typ(\cO)$ is an orbit of $N/K$-action,  $k_j$ is a divisor of $[N:K]$.
Recall that $[N:K]$ is a power of 2.
Hence $p\nmid k_j$ and $p^{e_p}\mid\prod_{l\neq j}w(y_l)$ for $j=1,  \ldots,  t_\cO$.
From this we see that $p^{e_p}\mid\gcd\{\prod_{j=1}^{t_\cO}w(y_j),  \mass(\cO)\prod_{j=1}^{t_\cO}w(y_j)\}$.
On the other hand,  the numerator of $\mass(\cO)$ equals  
    \[
    \frac{\mass(\cO)\prod_{j=1}^{t_\cO}w(y_j)}
    {\gcd\{\prod_{j=1}^{t_\cO}w(y_j),  \mass(\cO)\prod_{j=1}^{t_\cO}w(y_j)\}}
    \]
which is divisible by $p$ from the assumption.
Therefore we get $p^{e_p+1}\mid\mass(\cO)\prod_{j=1}^{t_\cO}w(y_j)$.
\end{proof}

\begin{remark}
From the proof,  one sees that \cref{lem:cong} holds for general order $\cO$ if $p$ is a (not necessarily odd) prime such that $p\mid\mass(\cO)$ and $p\nmid[N_\cO:K_\cO]$.
\end{remark}

Let $\{\phi_i\}_{i=1}^{t_\cO}$ be a basis of $\cS_N(\cO)$ consisting of normalized Hecke eigenforms.
We immediately obtain the following non-vanishing result from \cref{lem:cong}.

\begin{prop}\label{prop:non-vanish}
Assume that $\cS_N(\cO)$ is spanned by $\{\phi_i^\sigma\}_{\sigma\in\Gal(F_{N,  \cO}/\Q)}$,  where $\phi_i^\sigma(x)=\phi_i(x)^\sigma$.
Take $E\in X(D)$ and an order $\fo$ in $E$ so that $\Emb(\fo,  \cO)\neq\emptyset$.
We fix $\iota_0\in\Emb(\fo,  \cO)$.
If there is an odd prime factor $p$ of $\mass(\cO)$ with $p\nmid h_\fo$,  then $\fP_{\iota_0,  \fo}(\phi_i)\neq0$ for any $i$.
\end{prop}

\begin{proof}
By \cref{lem:cong},  we can take $\varphi\in\cS_N(\cO)$ which takes values in $1+p\Z$.
Then we have $u_\fo\fP_{\iota_0,  \fo}(\varphi)\equiv h_\fo\not\equiv0 \pmod p$,
hence $\fP_{\iota_0,  \fo}(\varphi)\neq0$.
If we write $\varphi=\sum_\sigma a_\sigma\phi_i^\sigma$ with $a_\sigma\in\C$,  
then
    \[
    \fP_{\iota_0,  \fo}(\varphi)=\sum_{\sigma\in\Gal(F_{N,  \cO}/\Q)}
    a_\sigma\fP_{\iota_0,  \fo}(\phi_i^\sigma)
    =\sum_{\sigma\in\Gal(F_{N,  \cO}/\Q)}
    a_\sigma\fP_{\iota_0,  \fo}(\phi_i)^\sigma.
    \]
Hence $\fP_{\iota_0,  \fo}(\phi_i)^\sigma\neq0$ for at least one $\sigma\in\Gal(F_{N,  \cO}/\Q)$.
This completes the proof.
\end{proof}

Hereafter,  we focus on the situation of \cref{prop:non-vanish} with $p=3$.
It is convenient to record the assumption in \cref{prop:non-vanish} for later use.
\begin{condition}\label{condition:non-vanish}
Let the notation be as above.
\begin{itemize}
\item[(a)] The vector space $\cS_N(\cO)$ is spanned by $\{\phi_i^\sigma\}_{\sigma\in\Gal(F_{N,  \cO}/\Q)}$ for some $i$,  where $\phi_i^\sigma(x)=\phi_i(x)^\sigma$.
\item[(b)] The total mass $\mass(\cO)$ is divisible by 3.
\end{itemize}
\end{condition}

\example\label{ex:19-1}
Suppose $\disc(D)=\disc(\cO)=19$.
In this case,  the type number is $t_\cO=2$. 
Hence $\dim\cS_N(\cO)=1$ and $F_{N, \cO}=\Q$. 
Take a non-zero vector $\phi\in\cS_N(\cO)$.
By the Eichler mass formula \cite[Theorem 25.3.19]{Voight},  $\mass(\cO)=\frac32$.
Thus \cref{condition:non-vanish} holds and we can apply \cref{prop:non-vanish} with $p=3$ to see that $\fP_E(\phi)\neq0$ for $E\in X(D)$ satisfying $\Emb(\fo_E,  \cO)\neq\emptyset$ and $3\nmid h_E$.
\medskip

For a finite set $S$ of places of $\Q$ and $\{\cE_v\}_{v\in S}\in\prod_{v\in S}X_v$,  let $X(\{\cE_v\}_{v\in S})$ denote the set of $E\in X$ satisfying $E_v\simeq \cE_v$ for all $v\in S$. 
The goal of this section is to prove the following two results.

\begin{theorem}\label{thm:PeriodGoldfeld}
Suppose that \cref{condition:non-vanish} holds.
Let $\phi\in\cS_N^\new(\cO)$ be a Hecke eigenform,  $\pi$ the irreducible cuspidal automorphic representation of $G_\A$ generated by $\phi$ and $\pi'=\otimes_v\pi'_v$ its Jacquet-Langlands transfer to $\PGL_2(\A)$.
We take  $\{\cE_v\}_{v\in S_\cO}\in\prod_{v\in S_\cO}X(D_v)$ so that the ramification set of $D$ coincides with $\{v\in S_\cO \mid \varepsilon(\pi'_v; \cE_v)=-1\}$.
Then $\cP_E(\phi)\neq0$ for positive proportion of $E\in X(\{\cE_v\}_{v\in S_\cO})$,  i.e.    
    \[
    \#\{E\in X(\{\cE_v\}_{v\in S_\cO}) \mid |\Delta_E|<x,  \ 
    \cP_E(\phi)\neq0\} \gg x
    \]
when $x\to\infty$.
Here,  $\cP_E(\phi)\neq0$ means $\Emb(\fo_E,  \cO)\neq\emptyset$ and $\cP_{\iota,  E}(\phi)\neq0$ for some (any) $\iota\in\Emb(\fo_E,  \cO)$.
\end{theorem}

\begin{theorem}\label{thm:AutomGoldfeld}
Suppose that \cref{condition:non-vanish} holds.
Let $\pi$ be the irreducible cuspidal automorphic representation of $G_\A$ generated by a Hecke eigenform in $\cS_N(\cO)$ and $\pi'$ its Jacquet-Langlands transfer to $\PGL_2(\A)$.
Then
    \[
    \lim_{x\to\infty}
    \frac{\#\{E\in X\mid -x< \Delta_E<0,  \ L(\tfrac12,  \pi'\otimes\eta_E)\neq0\}}
    {\#\{E\in X\mid -x<\Delta_E<0\}} \
    \geq \ \frac12 \ \prod_{p\mid\disc(\cO)}n_p,
    \]
where $p$ runs through prime factors of $\disc(\cO)=\disc(D)\level(\cO)$ and
    \[
    n_p=
        \begin{cases}
        \frac{p+2}{2(p+1)} & \text{if $p\neq2$} \\
        \frac{1}{24}& \text{if $p=2$}.
        \end{cases}
    \]
\end{theorem}

\begin{remark}
We do not know whether there are infinitely many cases that \cref{thm:PeriodGoldfeld} and \cref{thm:AutomGoldfeld} cover (see \cite[Conjecture A]{Martin2}).
But one can see that there are many examples which satisfy \cref{condition:non-vanish} and this assumption is not restrictive.
For example,  the following is the list of prime numbers $p\leq10^4$ such that a maximal order $\cO$ of $\disc(\cO)=p$ in a definite quaternion algebra over $\Q$ satisfies \cref{condition:non-vanish}.
There are 150 such primes (note that the number of maximal orders $\cO$ of prime discriminant $\disc(\cO)=p\leq10^4$ satisfying the condition (b) is 203).

19, 37, 127, 163, 181, 271, 379, 523, 541, 613, 631, 757,  811, 829, 883, 919, 937, 991, 1009, 1117, 1279, 1423, 1459, 1549, 1657, 1747, 1783, 1801, 2017, 2053, 2161, 2179, 2269, 2287, 2377, 2467, 2503, 2521, 2539, 2557, 2647, 2683, 2719, 2791, 2971, 3061, 3079, 3169, 3187, 3457, 3511, 3529, 3637, 3673, 3691, 3709, 3727, 3853, 3889, 4051, 4159, 4177, 4231, 4447, 4519, 4591, 4663, 4789, 4861, 4933, 4969, 4987, 5023, 5059, 5077, 5113, 5167, 5437, 5527, 5563, 5581, 5653, 
5743, 5779, 5851, 5869, 5923, 6121, 6229, 6247, 6301, 6373, 6427, 6481, 6553, 6607, 6661, 6679, 6733, 6823, 6841, 6967, 7039, 7129, 7219, 7237, 7309, 7417, 7489, 7507, 7561, 7687, 7741, 7759, 7993, 8011, 8101, 8191, 8209, 8263, 8317, 8353, 8389, 8443, 8461, 8623, 8641, 8677, 8713, 8731, 8803, 8821, 8839, 8893, 8929, 9001, 9091, 9109, 9181, 9199, 9343, 9397, 9433, 9613, 9631, 9649, 9721, 9739, 9883, 9973. 

\end{remark}

\subsection{Goldfeld's conjecture}

For the purpose of comparison,  we recall the conjecture of Goldfeld \cite{Goldfeld} on elliptic curves and the weaker version of it.
We also introduce related conjectures for automorphic $L$-functions and toric periods. 
\cref{thm:PeriodGoldfeld} and \cref{thm:AutomGoldfeld} provide evidence for these conjectures.
Let $C$ be an elliptic curve over $\Q$ and $L(s,  C)$ its $L$-function.
For a quadratic field $E$,  let $C_E$ denote the quadratic twist of $C$ by $\Delta_E$.
\begin{conj}[Goldfeld]\label{conj:Goldfeld}
For an elliptic curve $C$ over $\Q$, 
    \[
    \lim_{x\to\infty}
    \frac{\#\{E\in X\mid |\Delta_E|<x,  \ L(1,  C_E)\neq0\}}
    {\#\{E\in X\mid |\Delta_E|<x\}}=\frac12.
    \]
\end{conj}
The following is a weaker version of this conjecture.
See \cite[Conjecture 1.2]{KL},  for example.
\begin{conj}[Weak Goldfeld]\label{conj:WeakGoldfeld}
For an elliptic curve $C$ over $\Q$,
    \[
    \#\{E\in X\mid |\Delta_E|<x,  \ L(1,  C_E)\neq0\}\gg x
    \]
when $x\to\infty$.
\end{conj}

Since $L$-functions of elliptic curves are automorphic $L$-functions,  one can expect the following.
\begin{conj}[Automorphic (Weak) Goldfeld]\label{conj:AutomGoldfeld}
Let $\pi'$ be an irreducible cuspidal automorphic representation of $\PGL_2(\A)$, 
\begin{itemize}
\item[(1)] The twisted $L$-value $L(\tfrac12,  \pi'\otimes\eta_E)$ does not vanish for 50\% of $E\in X$,  namely:      
    \begin{equation}\label{eq:AutomGoldfeld}
    \lim_{x\to\infty}
    \frac{\#\{E\in X\mid |\Delta_E|<x,  \ L(\tfrac12,  \pi'\otimes\eta_E)\neq0\}}
    {\#\{E\in X\mid |\Delta_E|<x\}}=\frac12.
    \end{equation}
\item[(2)] We have $L(\tfrac12,  \pi'\otimes\eta_E)\neq0$ for positive proportion of $E\in X$,  i.e.  
    \begin{equation}\label{eq:WeakAutomGoldfeld}
    \#\{E\in X\mid |\Delta_E|<x,  \ L(\tfrac12,  \pi'\otimes\eta_E)\neq0\}\gg x
    \end{equation}
when $x\to\infty$.
\end{itemize}
\end{conj}
From \cref{thm:Waldspurger},  we see that the following conjecture implies \eqref{eq:WeakAutomGoldfeld} if $L(\tfrac12,  \pi')\neq0$.
\begin{conj}[Weak Goldfeld for Toric Periods]\label{conj:WeakPeriodGoldfeld}
We remove the assumption that $D$ is definite.
Let $\pi$ be an irreducible cuspidal automorphic representation of $G_\A$ with $L(\tfrac12,  \pi)\neq0$ and $\pi'=\otimes_v\pi'_v$ its Jacquet-Langlnds transfer to $\PGL_2(\A)$.
We take $\{\cE_v\}_{v\in S_\cO}\in \prod_{v\in S_\cO}X(D_v)$ so that the ramification set of $D$ coincides with $\{v\in S_\cO \mid \varepsilon(\pi'_v; \cE_v)=-1\}$.
Then $\cP_E\not\equiv0$ on $\pi$ for positive proportion of $E\in X(\{\cE_v\}_{v\in S_\cO})$,  i.e.  
    \begin{equation}\label{eq:WeakPeriodGoldfeld}
    \#\{E\in X(\{\cE_v\}_{v\in S_\cO})\mid |\Delta_E|<x,  \ \cP_E\not\equiv0  
    \text{ \rm on $\pi$}\}
    \gg x
    \end{equation}
when $x\to\infty$.
\end{conj}

\begin{remark}
Obviously,  the non-vanishing of the toric period of a fixed automorphic form we consider in \cref{thm:PeriodGoldfeld} is stronger than the non-vanishing of the linear form $\cP_E$ on $\pi$ which we consider in \cref{conj:WeakPeriodGoldfeld} and \cref{conj:PeriodGoldfeld}.
We do not know how to formulate such a stronger non-vanishing problem for general cuspidal automorphic representations.
\end{remark}

The following diagram illustrates the relation among these conjectures.
    \[
    \xymatrix @R=10pt @C=50pt{
      \setlength{\fboxsep}{0pt}
    &
    \setlength{\fboxsep}{0pt}
    \fbox{\begin{minipage}{115pt}
    \begin{tabular}{c}
    Automorphic Goldfeld \\[-3pt]
    (\cref{conj:AutomGoldfeld} (1))
    \end{tabular}
    \end{minipage}}\ar @{=>}[d] \ar @{=>}[r]&
    \setlength{\fboxsep}{0pt}
    \fbox{\begin{minipage}{87pt}
    \begin{tabular}{c}
    Goldfeld \\[-3pt]
    (\cref{conj:Goldfeld})
    \end{tabular}
    \end{minipage}} \ar @{=>}[d]       \\
    \setlength{\fboxsep}{0pt}
    \fbox{\begin{minipage}{95pt}
    \begin{tabular}{c}
    Weak Goldfeld \\[-3pt]
    for Toric Periods \\[-3pt]
    (\cref{conj:WeakPeriodGoldfeld})
    \end{tabular}
    \end{minipage}} \ar @{=>}^(0.47){\text{if $L(\tfrac12,  \pi)\neq0$}}[r]&
    \setlength{\fboxsep}{0pt}
    \fbox{\begin{minipage}{110pt}
    \begin{tabular}{c}
    Automorphic \\[-3pt]
    Weak Goldfeld \\[-3pt]
    (\cref{conj:AutomGoldfeld} (2))
    \end{tabular}
    \end{minipage}} \ar @{=>}[r] & 
    \setlength{\fboxsep}{0pt}
    \fbox{\begin{minipage}{93pt}
    \begin{tabular}{c}
    Weak Goldfeld \\[-3pt]
    (\cref{conj:WeakGoldfeld})
    \end{tabular}
    \end{minipage}} 
    }\]

Now return to a Hecke eigenform $\phi\in\cS_N^\new(\cO)$,  in particular $D$ is definite.
We formulate an analogue of \cref{conj:Goldfeld} for toric periods.
Let $\pi$ be an irreducible cuspidal automorphic representation of $G_\A$ generated by $\phi$ and $\pi'=\otimes_v\pi'_v$ its Jacquet-Langlands transfer to $\PGL_2(\A)$.
We take $\{\cE_v\}_{v\in S_\cO}\in \prod_{v\in S_\cO}X(D_v)$ so that the ramification set of $D$ equals $\{v\in S_\cO \mid \varepsilon(\pi'_v; \cE_v)=-1\}$.

\begin{conj}[Goldfeld for toric periods]\label{conj:PeriodGoldfeld}
Let $\pi=\otimes_v\pi_v$ and $\{\cE_v\}_{v\in S_\cO}$ be as above.
Assume that $L(\tfrac12,  \pi)\neq0$.
The toric period $\cP_E$ is non-zero on $\pi$ for 100\% of $E\in X(\{\cE_v\}_{v\in S_\cO})$,  namely:

    \begin{equation}\label{eq:PeriodGoldfeld}
    \lim_{x\to\infty}
    \frac{\#\{E\in X(\{\cE_v\}_{v\in S_\cO})\mid |\Delta_E|<x,  
    \ \cP_E\not\equiv0  \text{ \rm on $\pi$}\}}
    {\#\{E\in X(\{\cE_v\}_{v\in S_\cO})\mid |\Delta_E|<x\}}=1.
    \end{equation}
\end{conj}

One can check that \cref{conj:PeriodGoldfeld} follows from \cref{conj:WeakPeriodGoldfeld}.

\begin{prop}
Let $\pi=\otimes_v\pi_v$ and $\{\cE_v\}_{v\in S_\cO}$ be as above.
Suppose that $L(\tfrac12,  \pi)\neq0$ and $\pi'$ satisfies \eqref{eq:AutomGoldfeld}.
Then \eqref{eq:PeriodGoldfeld} holds for $\pi$ and $\{\cE_v\}_{v\in S_\cO}$.
\end{prop}
 
\begin{proof}
For a finite set $S$ of finite places of $\Q$ and $\{\cE'_v\}_{v\in S}\in \prod_v X_v$,  the following is well-known:
    \[
    \lim_{x\to\infty}\frac{\#\{E\in X(\{\cE'_v\}_{v\in S}) \mid 0<\Delta_E<x\}}
    {\#\{E\in X(\{\cE'_v\}_{v\in S}) \mid -x<\Delta_E<0\}}=1.
    \]
Note that $\varepsilon(\pi'_\infty; \R\times\R)=-\varepsilon(\pi'_\infty; \C)$ since $\pi'_\infty\simeq\pi'_\infty\otimes\sgn$,  where $\sgn$ is the sign character on $\GL_2(\R)$.
Hence we see that
    \[
    \lim_{x\to\infty}\frac{\#\{E\in X \mid |\Delta_E|<x, \ \varepsilon(\pi'; E)=1\}}
    {\#\{E\in X \mid |\Delta_E|<x, \ \varepsilon(\pi'; E)=-1\}}=1.
    \]
From this it follows that \eqref{eq:AutomGoldfeld} for $\pi'$ is equivalent to 
\begin{equation}\label{eq:PeriodGoldfeld-2}
    \lim_{x\to\infty}
    \frac{\#\{E\in X\mid |\Delta_E|<x,  \ L(\tfrac12,  \pi'\otimes\eta_E)\neq0\}}
    {\#\{E\in X \mid |\Delta_E|<x, \ \varepsilon(\pi'; E)=1\}}=1
\end{equation}
Since $\varepsilon(\pi'_v; \cE'_v)=1$ for any $v\not\in S_\cO$ and $\cE_v\in X_v$,  the set in the denominator decomposes as 
\begin{equation}\label{eq:PeriodGoldfeld-3}
    \coprod_{\substack{\{\cE'_v\}_{v\in S_\cO} \\
    \prod_{v\in S_\cO}\varepsilon(\pi'_v;\cE'_v)=1}}
    \{E\in X(\{\cE'_v\}_{v\in S_\cO}) \mid |\Delta_E|<x\}.
\end{equation}
Similarly we have a decomposition of the set in the numerator:
    \[
    \coprod_{\substack{\{\cE'_v\}_{v\in S_\cO} \\
    \prod_{v\in S_\cO}\varepsilon(\pi'_v;\cE_v)=1}}
    \{E\in X(\{\cE'_v\}_{v\in S_\cO}) \mid |\Delta_E|<x,  \ 
    L(\tfrac12,  \pi'\otimes\eta_E)\neq0\}.
    \]
Thus we can rewrite \eqref{eq:PeriodGoldfeld-2} as
\begin{align*}
    1=\lim_{x\to\infty}\sum_{\substack{\{\cE'_v\}_{v\in S_\cO} \\
    \prod_{v\in S_\cO}\varepsilon(\pi'_v;\cE'_v)=1}}
    &\frac{\#\{E\in X(\{\cE'_v\}_{v\in S_\cO})\mid 
    |\Delta_E|<x,  \ L(\tfrac12,  \pi'\otimes\eta_E)\neq0\}}
    {\#\{E\in X(\{\cE'_v\}_{v\in S_\cO}) \mid |\Delta_E|<x\}} \\
    &\hspace{80pt} \times\frac{\#\{E\in X(\{\cE'_v\}_{v\in S_\cO})\mid 
    |\Delta_E|<x\}}
    {\#\{E\in X \mid |\Delta_E|<x,  \ \varepsilon(\pi'; E)=1\}}.
\end{align*}
From \eqref{eq:PeriodGoldfeld-3},  this is equivalent to
    \[
    \lim_{x\to\infty}
    \frac{\#\{E\in X(\{\cE'_v\}_{v\in S_\cO})\mid 
    |\Delta_E|<x,  \ L(\tfrac12,  \pi'\otimes\eta_E)\neq0\}}
    {\#\{E\in X(\{\cE'_v\}_{v\in S_\cO}) \mid |\Delta_E|<x\}}=1.
    \]
It follows from \cref{thm:Waldspurger} (1) that for $\{\cE_v\}_{v\in S_\cO}$,  the set in the numerator equals
    \[
    \{E\in X(\{\cE_v\}_{v\in S_\cO})\mid 
    |\Delta_E|<x,  \ \cP_E \not\equiv0 \text{ on $\pi$}\}.
    \]
This completes the proof.
\end{proof}

\example\label{ex:19-2}
Suppose $\disc(D)=\disc(\cO)=19$.
In this case,  \cref{condition:non-vanish} is satisfied as we have seen in \cref{ex:19-1}.
Let $\pi$ be the cuspidal automorphic representation generated by a non-zero element of $\cS_N(\cO)=\cS_N^\new(\cO)$.
Then \cref{thm:AutomGoldfeld} shows 
    \[
    \lim_{x\to\infty}
    \frac{\#\{E\in X\mid -x< \Delta_E<0,  \ L(\tfrac12,  \pi'\otimes\eta_E)\neq0\}}
    {\#\{E\in X\mid -x<\Delta_E<0\}} \
    \geq \ \frac{21}{80}.
    \]
Consider the elliptic curve $C : y^2+y=x^3+x^2-9x-15$ (19a1 in Cremona's labeling).
We see that $L(\tfrac12,  \pi\otimes\eta_E)\neq0$ is equivalent to $L(1,  C_E)\neq0$ for any quadratic field $E$.
Thus,  more than $\frac{21}{80}=26.25\%$ of imaginary quadratic twists of $C$ have analytic rank 0 (compare the lower bound $\frac{19}{120}=15.833\%$ in \cite[Example 9.9]{KL}).

\example\label{ex:65.2.a.c}
Suppose $\disc(D)=13$ and $\disc(\cO)=65$.
In this case,  the type number is $t_\cO=3$,  $\cS_N(\cO)=\cS_N^\new(\cO)$ is 2-dimensional and $F_{N,  \cO}=\Q(\sqrt{3})$.
By the Eichler mass formula \cite[Theorem 25.3.19]{Voight},  $\mass(\cO)=6$.
Thus \cref{condition:non-vanish} holds and \cref{thm:PeriodGoldfeld} shows that for a Hecke eigenform $\phi\in\cS_N(\cO)$,  we have $\cP_E(\phi)\neq0$ for a positive proportion of  imaginary quadratic fields $E\in X(D)$.
We can also apply \cref{thm:AutomGoldfeld} to obtain
    \[
    \lim_{x\to\infty}
        \frac{\#\{E\in X\mid -x< \Delta_E<0,  \ L(\tfrac12,  \pi'\otimes\eta_E)\neq0\}}
    {\#\{E\in X\mid -x<\Delta_E<0\}} \
    \geq \ \frac{15}{96}.
    \]
Here,  $\pi'$ is the Jacquet-Langlands transfer to $\PGL_2(\A)$ of the  representation of $G_\A$ generated by $\phi$.
Hence,  more than $\frac{15}{96}=15.625\%$ of imaginary quadratic twists of $\pi'$ have non-vanishing central $L$-values.

\begin{remark}
We obtain a lower bound toward Godlfeld's conjecture  for elliptic curves (\cref{conj:Goldfeld}) if $F_{N,  \cO}=\Q$ (equivalently,  $t_\cO=2$) and $3\mid\mass(\cO)$.
According to Kirschmer's data base \cite{Kirschmer},  there are 29 such quaternion orders.
One can check that all of the corresponding elliptic curves have 3-isogenies and hence the weak Godldfeld conjecture (\cref{conj:WeakGoldfeld}) is already verified in these cases (see \cite[Theorem 1.5]{KL}).

However,  our result improves the lower bounds for several elliptic curves as shown in \cref{ex:19-2}. 
Note also that \cref{thm:PeriodGoldfeld} provides new examples for which \cref{conj:AutomGoldfeld} (2) and \cref{conj:WeakPeriodGoldfeld} are valid as shown in \cref{ex:65.2.a.c}.
\end{remark}

\subsection{Proof of \cref{thm:PeriodGoldfeld} and \cref{thm:AutomGoldfeld}}
\label{subsec:ProofGoldfeld}

Write $\Typ(\cO)=\{[\cO_1],  \ldots [\cO_{t_\cO}]\}$ with $\cO_1=\cO$.
In what follows,  we need a sufficient condition for existence of optimal embeddings. 
The next lemma is a straightforward consequence of \cref{lem:local_emb} and the argument in \cref{subsec:Action}.

\begin{lemma}\label{lem:optimal-1}
A quadratic field $E\in X(D)$ has an optimal embedding with respect to some $\cO_j$ if all prime factors of $\level(\cO)$ split or ramify in $E$.
\end{lemma}

\cref{lem:optimal-1} provides a sufficient condition for fixed $E\in X(D)$ to have an optimal embedding with respect to one of $\{\cO_j\}_{j=1}^{t_\cO}$.
The next lemma asserts that under a certain condition,  all but finitely many $E\in X(D)$ have an optimal embedding with respect to a fixed order $\cO$.
This is known by \cite[Theorem 10]{Michel}.  
See also \cite[Lemma 8]{SP}.

\begin{lemma}\label{lem:Duke}
Let $\{E_k\}_{k=1}^\infty$ be a sequence of imaginary quadratic fields in $X(D)$.
Assume that for each $k$,  $E_k$ has an optimal embedding with respect to $\cO_{j_k}$ for some $j_k$.
Then there exists $X>0$ such that $E_k$ has an optimal embedding with respect to $\cO$ for all $k>X$.
\end{lemma}

\begin{proof}
We fix $\iota_k\in\Emb(\fo_{E_k},  \cO_{j_k})$ for each $k$.
For $\phi\in\cA(\cO)$,  the sequence $\{\cP_{\iota_k,  E_k}(\phi)\}_{k=1}^\infty$ converges to $\int_{G(\Q)\bs G(\A)}\phi(g)\d g$, were $\d g$ is the Tamagawa measure on $G(\A)$.
This is a variant of Duke's theorem \cite{Duke}.
Here,  for completeness we record a short proof.

Let $\{\phi_i\}_{i=1}^{h_\cO}$ be an orthogonal basis of $\cA(\cO)$ consisting of Hecke eigenforms and $\pi_j$ the irreducible automorphic representation of $G_\A$ generated by $\phi_j$.
Note that $\pi_i\not\simeq\pi_{i'}$ for $i\neq i'$.
We may and will assume $\phi_1\equiv 1$ is the constant function.
It is obvious that $\cP_{\iota_k,  E_k}(\phi_1)=2$ for any $k$.

If $\pi_i$ is 1-dimensional,  it comes from a quadratic character on $\A^\times$.
Thus we have $\cP_{\iota_k,  E_k}(\phi_i)=0$ for all but at most one $E_k$ if $\pi_i$ is 1-dimensional and $i\neq1$.

If $\pi_i$ is not 1-dimensional,  it corresponds with a cuspidal automorphic representation of $\PGL_2(\A)$.
Hence the subconvex bound \cite[Theorem 2]{BH} combined with \cref{thm:Waldspurger} (2) and \cref{lem:bound_period} shows that $\lim_{k\to\infty}|\cP_{\iota_k,  E_k}(\phi_i)|^2 \ll \lim_{k\to\infty}|\Delta_{E_k}|^{-\beta}=0$ for some $\beta>0$.
Note that we have
    \[
    \int_{G(\Q)\bs G(\A)}\phi_i(g)\d g=
        \begin{cases}
        2 & \text{if $i=1$,} \\
        0 & \text{otherwise.}
        \end{cases}
    \]
Since any $\phi\in \cA(\cO)$ is expressed as $\phi=\sum_{i=1}^{h_\cO} \frac{\langle \phi,  \phi_i \rangle}{\langle \phi_i,  \phi_i \rangle}\phi_i$,  we get the desired assertion.
Moreover,  it is obvious from the above argument that this is uniform convergence for $\phi\in\cA(\cO)$ with $\langle \phi,  \phi \rangle <1$. 

We can deduce the lemma from Duke's theorem as follows.
For each $k$,  let $i_k\,\colon\Cl(\fo_{E_k}) \rightarrow \Cl(\cO_{j_k})$ be the map induced from $\iota_k$.
We also fix $y_k=(y_{k,  v})_v\in G_{\A_f}$ such that $\cO_{j_k,  v}=y_{k,  v}\cO_vy_{k,  v}^{-1}$ and let $\widetilde{i}_k$ denote the composition of $i_k$ with $\Cl(\cO_{j_k}) \xrightarrow{\sim} \Cl(\cO)$ induced from the right translation by $y_k$.
From \eqref{eq:period} and $\cP_{\iota_k,  E_k}(\phi)=2u_{E_k}h_{E_k}^{-1}\fP_{\iota_k,  E_k}(\phi)$,  we see that for any $\varepsilon>0$,  there exists $X>0$ such that for $k>X$ and $\phi\in\cA(\cO)$ with $\langle \phi,  \phi \rangle <1$, 
    \[
    \left|\sum_{[I]\in\Cl(\cO)}\left(\frac{\#(\widetilde{i}_k^{-1}([I]))}{h_{E_k}}
    -\frac{1}{w(x_{[I]})}\right)\phi([I])\right|<\varepsilon.
    \]
For each $[I]\in\Cl(\cO)$,  as we may choose $\phi\in\cA(\cO)$ so that $\phi([J])=0$ if $[J]\neq [I]$,  it follows that $\lim_{k\to\infty} h_{E_k}^{-1}\#(\widetilde{i}_k^{-1}([I]))=w(x_{[I]})^{-1}$.
Hence $\widetilde{i}_k^{-1}([I])\neq\emptyset$ for any $[I]\in\Cl(\cO)$ and sufficiently large $k$.
In particular $\Emb(\fo_{E_k},  \cO_j)\neq\emptyset$ for any $j=1,  \ldots,  t_\cO$ and sufficiently large $k$ as we have seen in \cref{subsec:Action}.
\end{proof}

\begin{remark}
\cref{lem:Duke} holds for any (not necessarily Eichler) order $\cO$ in $D$.
\end{remark}

Combining \cref{lem:optimal-1} with \cref{lem:Duke} we obtain the following.

\begin{cor}\label{cor:optimal-2}
A quadratic field $E\in X(D)$ has an optimal embedding with respect to $\cO$ if $|\Delta_E|$ is sufficiently large and all prime factors of $\level(\cO)$ split or ramify in $E$.  

In particular,  $E\in X(D)$ has an optimal embedding with respect to any maximal order in $D$ if $|\Delta_E|$ is sufficiently large.
\end{cor}

We need the following form of an existence result of optimal embeddings to prove \cref{thm:PeriodGoldfeld}.

\begin{cor}\label{cor:optimal-3}
We take  $\{\cE_v\}_{v\in S_\cO}\in\prod_{v\in S_\cO}X(D_v)$ so that the ramification set of $D$ coincides with $\{v\in S_\cO \mid \varepsilon(\pi'_v; \cE_v)=-1\}$.
Then,  all but finitely many $E\in X(\{\cE_v\}_{v\in S_\cO})$ has an optimal embedding with respect to $\cO$.
\end{cor}

\begin{proof}
From \cref{cor:optimal-2},  it suffices to show that for a prime factor $p$ of $\level(\cO)$,  $\cE_p\simeq\Q_p\times\Q_p$ or $\cE_p$ is a ramified extension of $\Q_p$. 
Let $\St_p$ denote the Steinberg representation of $G_p\simeq\PGL_2(\Q_p)$ and $\omega_p$ be the composition of the non-trivial unramified quadratic character of $\Q_p^\times$ with the reduced norm on $D_p^\times$.
We regard $\omega_p$ as a character of $G_p$.
Then $\pi'_p\simeq\St_p$ or $\St_p\otimes\omega_p$.
We see that for $\cE_p'\in X_p$,  $\varepsilon(\St_p; \cE_p')=1$ if and only if $\cE_p'\simeq\Q_p\times\Q_p$
and $\varepsilon(\St_p\otimes\omega_p; \cE_p')=-1$ if and only if $\cE_p'$ is the unramified quadratic extnesion of $\Q_p$.
By the choice of $\{\cE_v\}_{v\in S_\cO}$,  we have $\varepsilon(\pi'_p; \cE_p)=1$ for any prime $p \mid \level(\cO)$.
Therefore we see that $\cE_p\simeq\Q_p\times\Q_p$ or $\cE_p$ is a ramified extension of $\Q_p$.
\end{proof}

Hartung \cite{Hartung} showed that there are infinitely many imaginary quadratic fields $E$ with $3\nmid h_E$.
Hartung's result was extended by many researchers with various congruence conditions on the discriminants,  giving explicit lower bounds for the proportion of such quadratic fields.
For more details,  see \cite[\S\,9]{KL} and the references thereof.
We need a variant of those results to obtain an explicit lower bounds for the proportion of $E\in X(D)$ with $\Emb(\fo_E,  \cO')\neq\emptyset$ for some $[\cO']\in\Typ(\cO)$ and $3 \nmid h_E$.
First,  we recall the notion of  a valid pair from \cite{KL}.

\begin{definition}
A pair $(m,  M)$ of positive integers is \textit{valid} if it satisfies the following properties:
\begin{itemize}
\item for an odd prime factor $\ell$ of $\gcd(m,  M)$,   $\ell^2\mid M$ and $\ell^2\nmid m$;
\item if $M$ is even,  then one of the following holds:
\begin{itemize}
\item[(i)] $4\mid M$ and $m\equiv 1 \pmod 4$,
\item[(ii)] $16\mid M$ and $m\equiv 8 \text{ or } 12 \pmod {16}$.
\end{itemize}
\end{itemize}
\end{definition}

The following is \cite[Proposition 9.3]{KL},  which is attributed to \cite{Taya} there.

\begin{prop}\label{prop:density-1}
For a valid pair $(m,  M)$,  we have
    \[
    \lim_{x\to\infty}
    \frac{\#\{E\in X \mid -x<\Delta_E<0,  \ \Delta_E \equiv m \pmod M ,  \ 
    3\nmid h_E\}}
    {\#\{E\in X \mid -x<\Delta_E<0\}}
    \geq \frac{1}{2\Phi(M)}\prod_{p\mid M}\frac{q(p)}{p+1}.
    \]
Here,  $\Phi$ is the Euler totient function,  $p$ runs over prime factors of $M$ and $q(p)=
    \begin{cases}
    4 & \text{if $p=2$}\\
    \ell & \text{otherwise}.
    \end{cases}$
\end{prop}

\begin{cor}\label{cor:density-2}
For $x>0$,  let $N(x,  \cO)$ denote the set of $E\in X(D)$ with the property that  $|\Delta_E|<x$,  $3\nmid h_E$ and all prime factors of $\level(\cO)$ split or ramify in $E$.
Then
    \[
    \lim_{x\to\infty}\frac{\# N(x,  \cO)}{\#\{E\in X \mid -x<\Delta_E<0\}}
    \geq \frac12\prod_{p\mid\disc(\cO)} n_p,
    \]
where $p$ runs through prime factors of $\disc(\cO)=\disc(D)\level(\cO)$, $n_p=\frac{p+2}{2(p+1)}$ if $p\neq2$ and $n_2=\frac{1}{24}$.
\end{cor}

\begin{proof}
Set $M:=\disc(\cO)^2k(\cO)$,  where $k(\cO):=4$ if $\disc(\cO)$ is even and $1$ otherwise.
By the Chinese remainder theorem,  the number of positive integers $m<M$ which satisfy the following property is not less than $k(\cO)^{-1}\prod_{\substack{p\mid M \\ p>2}}\frac{(p+2)(p-1)}{2}$,  where $p$ runs over odd prime factors of $\disc(\cO)$:
\begin{itemize}
\item $(m,  M)$ forms a valid pair; 
\item for $E\in X$ with $\Delta_E \equiv m \pmod M$,  prime factors of $\disc(D)$ do not split in $E$ and prime factors of $\level(\cO)$ split or ramify in $E$.
\end{itemize}
For such a pair $(m,  M)$,  one sees that the set $\{E\in X \mid -x<\Delta_E<0,  \ \Delta_E \equiv m \pmod M ,  \ 3\nmid h_E\}$ is contained in $N(x,  \cO)$.
Thus the assertion follows from \cref{prop:density-1}.
\end{proof}

Now we are ready to prove \cref{thm:PeriodGoldfeld} and \cref{thm:AutomGoldfeld}.

\begin{proof}[Proof of \cref{thm:PeriodGoldfeld}]
From \cref{cor:optimal-3},  all but finitely many $E\in X(\{\cE_v\}_{v\in S_\cO})$ have an optimal embedding with respect to $\cO$.
Thus \cref{prop:non-vanish} and \cref{cor:density-2} imply
    \[
    \lim_{x\to\infty}
    \frac{\#\{E\in X(\{\cE_v\}_{v\in S_\cO}) \mid |\Delta_E|<x, \  
    \cP_E(\phi)\neq0\}}
    {\#\{E\in X(\{\cE_v\}_{v\in S_\cO}) \mid |\Delta_E|<x,  \ 3\nmid h_E\}} \geq 1.
    \]
and $\#\{E\in X(\{\cE_v\}_{v\in S_\cO}) \mid |\Delta_E|<x,  \ 3\nmid h_E\} \gg x$ when $x\to\infty$.
This completes the proof.
\end{proof}

\begin{proof}[Proof of \cref{thm:AutomGoldfeld}]
Suppose that $E\in X(D)$ has an optimal embedding with respect to $\cO_j$ and $3\nmid h_E$.
We fix $\iota_j\in\Emb(\fo_E,  \cO_j)$.
Let $\phi\in\cS_N(\cO)$ be the Hecke eigenform which generates $\pi$.
Recall that we write $\Typ(\cO)=\{[\cO_1],  \ldots [\cO_{t_\cO}]\}$.
Take $y=(y_v)_v\in G_{\A_f}$ so that $\cO_{j,  v}=y_v\cO_vy_v^{-1}$ for any finite place $v$.
Set $\phi':=R_G(y)\phi$,  where $R_G$ is the right translation.
Then $\phi'$ is in $\in\cS_N(\cO_j)$ and it belongs to the space of $\pi$. 
It follows from \cref{prop:non-vanish} that $\cP_{\iota_j,  E}(\phi')\neq0$.
Thus \cref{thm:Waldspurger} (2) implies $L(\tfrac12,  \pi'\otimes\eta_E)\neq0$.
By \cref{lem:optimal-1},  this means that $\{E\in X \mid -x <\Delta_E<0,  \ L(\tfrac12,  \pi\otimes\eta_E)\neq0\}$ contains the set $N(x,  \cO)$ in \cref{cor:density-2}.
The theorem follows from \cref{cor:density-2}.
\end{proof}


\section{Sign changes}
\label{sec:Sign changes}

Let $\pi$ be an irreducible cuspidal automorphic representation of $G_\A$ generated by a normalized Hecke eigenform $\phi\in\cS_N^\new(\cO)$.
Throughout this section,  we assume that $L(\tfrac12,  \pi)\neq0$ and that $\cO$ is a maximal order \textit{i.e.} $\level(\cO)=1$.
Then $\fP_E(\phi)$ is well-defined for any $E\in X(D)$.
Since $\phi$ is normalized,  it takes values in the integer ring $\fo_\pi$ of its Hecke field $F_\pi$.
Note that $F_\pi$ is a totally real number field.
We fix an embedding $\iota_\pi\,\colon F_\pi \hookrightarrow \R$ and often regard $\phi$ as a real valued function by composing with $\iota_\pi$. 
We also fix a $\Z$-basis $\{v_i\}_i$ of $\fo_\pi$.
For $x\in F_\pi$,  write its expansion as $x=\sum_{i=1}^{[F_\pi : \Q]}x^{(i)} v_i$,  where $x^{(i)}\in\Q$.
Let $\phi^{(i)}$ be the function $[I]\mapsto\phi([I])^{(i)}$ on $\Cl(\cO)$.
Note that $\{\phi^{(i)} \mid i=1,  \ldots,  [F_\pi : \Q]\}$ is a basis of $\mathrm{Span}_\R\{\iota\circ\phi \mid \iota\in\Hom(F_\pi,  \R)\}$.
In particular,  each $\phi^{(i)}$ is an element of $\cS_N^\new(\cO)$.
The following is the main result of this section.

\begin{theorem}\label{thm:sign_change}
Keep the above notation.
\begin{itemize}
\item[(1)] The set $\{E\in X(D) \mid \fP_E(\phi)\neq0\}$ is infinite.
\item[(2)] The sequence $\{\fP_E(\phi)\in\R\}_{E\in X(D)}$ has infinitely many sign changes.
\item[(3)] The sequence $\{\fP_E(\phi)^{(i)}\in\Z\}_{E\in X(D)}$ has infinitely many sign changes for at least one $i$,  where $1 \leq i \leq [F_\pi : \Q]$.
\end{itemize}
\end{theorem}

\subsection{Hecke $L$-series}
\label{subsec:Hecke $L$-series}

Let $N$ and $k$ be positive integers. 
For a Dirichlet character $\chi$ on $\Z/4N\Z$,  let $S_{k+\frac12}(N,  \chi)$ be the space of weight $k+\frac12$ cusp forms on $\Gamma_0(4N)$ with Nebentypus $\chi$.
The Fourier expansion of $h\in S_{k+\frac12}(N,  \chi)$ has the form $h(z)=\sum_{n=1}^\infty a(n)e^{2\pi\sqrt{-1}nz}$.
We define the Hecke $L$-series $D(s,  h)$ of $h$ as the Mellin transform
    \[
    D(s,  h):=\int_0^\infty h(\sqrt{-1}y)y^{s-1}\d y
    =(2\pi)^{-s}\Gamma(s)\sum_{n=1}^\infty \frac{a(n)}{n^s}.
    \]
It follows from the Hecke bound $a(n)=O(n^{\frac{2k+1}{4}})$ \cite[(1.13)]{Shimura} that the last expression converges absolutely when $\re(s)>\frac{2k+5}{4}$.
The Fricke involution $h \mapsto h\,|[\tau_N]_{k+\frac12}$ is defined by
    \[
    (h\,|[\tau_N]_{k+\frac12})(z)=
    N^{-\frac{2k+1}{4}}(-\sqrt{-1}z)^{-\frac{2k+1}{2}}h\left(\frac{-1}{Nz}\right).
    \]
According to \cite[Proposition 1.4]{Shimura},  $h\,|[\tau_N]_{k+\frac12}$ belongs to $S_{k+\frac12}(N,  \bar{\chi}(\frac{N}{\cdot}))$,  where $(\frac{\cdot}{\cdot})$ is the Kronecker symbol and $(h\,|[\tau_N]_{k+\frac12})\,|[\tau_N]_{k+\frac12}=h$.
From
    \[
    D(s,  h)=\int_1^\infty h(\sqrt{-1}y)y^{s-1}\d y+N^{-s+\frac{2k+1}{4}}
    \int_1^\infty (h\,|[\tau_N]_{k+\frac12})(\sqrt{-1}y)y^{\frac{2k-1}{2}-s}\d y,
    \]
it follows that $D(s,  h)$ has holomorphic continuation to whole complex plane and satisfies the functional equation
    \[
    D(s,  h)=N^{-s+\frac{2k+1}{4}}D(k+\tfrac12-s,  h\,|[\tau_N]_{k+\frac12}).
    \] 

We let $h=\cW(\phi)\in S_{3/2}^+(\disc(\cO))$.
Recall that $a_\phi(n)$ is the $n$-th Fourier coefficient of $\cW(\phi)$.
From \eqref{eq:Fourier},  we get
\begin{equation}\label{eq:Fourier-3}
    a_\phi(n)=\sum_{\fo(-n)\subset\fo}u_\fo^{-1}
    \sum_{[I]\in\Cl(\cO)}\#\Big(\Emb(\fo,  \cO(I))_{/\sim}\Big)\phi([I]).
\end{equation}
In particular,  $a_\phi(n)=0$ unless $\Q(\sqrt{-n})\in X(D)$.
Hence we can rearrange the defining sum of $(2\pi)^s\Gamma(s)^{-1}D(s,  h)$ for $\re(s)>\frac74$ as 
\begin{equation}\label{eq:Hecke}
    (2\pi)^s\Gamma(s)^{-1}D(s,  h)=\sum_{n=1}^\infty \frac{a_\phi(n)}{n^s}
    =\sum_{E\in X(D)}|\Delta_E|^{-s}
    \sum_{m=1}^\infty \frac{a_\phi(m^2|\Delta_E|)}{m^{2s}}.
\end{equation}

The next lemma is an analogue of \cite[(3.2)]{BK}.
\begin{lemma}\label{lem:Hecke}
Let $L_\fin(s,  \pi)$ be the finite part of the standard $L$-function of $\pi$ and $L^{S_\cO}(s,  \eta_E)$ a partial Dirichlet $L$-function.
Then we have
    \[
    \sum_{m=1}^\infty \frac{a_\phi(m^2|\Delta_E|)}{m^s}
    =a_\phi(|\Delta_E|)
    \frac{L_\fin(s-\tfrac12,  \pi)}{L^{S_\cO}(s,  \eta_E)}.
    \]
\end{lemma}

\begin{proof}
Let $m$ be a positive integer.
Arguing as in \cref{sec:Classical},  we compute $a_\phi(m^2|\Delta_E|)$.
An order of $E$ contained in $\fo(m^2\Delta_E)$ is of the form $\fo=\fo(d^2\Delta_E)$ for some $d \mid m$.
From \cite[Proposition 30.5.3 (b)]{Voight},  we see  for $p \mid \disc(D)$,  $\Emb(\fo_p,  \cO_p)_{/\sim}=\emptyset$ if $\fo_p\neq\fo_{E,  p}$.
If this is the case,  $\Emb(\fo,  \cO(I))_{/\sim}=\emptyset$ for any $[I]\in\Cl(\cO)$ by the argument in \cref{subsec:Action}.
Thus \eqref{eq:Fourier-3} for $n=m^2\Delta_E$ becomes
\begin{equation}\label{eq:Fourier-4}
    a_\phi(m^2\Delta_E)=\sum_{\substack{d \mid m \\ (d,  \disc(D))=1}}
    u_{\fo(d^2\Delta_E)}^{-1}
    \sum_{[I]\in\Cl(\cO)}\#\Big(\Emb(\fo(d^2\Delta_E),  \cO(I))_{/\sim}\Big)\phi([I]).
\end{equation}

Let $\fo=\fo(d^2\Delta_E)$ and suppose $\Emb(\fo,  \cO(I))_{/\sim}\neq\emptyset$ for some $[I]\in\Cl(\cO)$.
By the argument in \cref{subsec:Dependence},  we may assume $\Emb(\fo,  \cO)\neq\emptyset$ replacing $\cO$ and $\phi$ if necessary.
Fix $\iota\in\Emb(\fo,  \cO)$.  
Note that $\#\Emb(\fo_p,  \cO_p)_{/\sim}=1$ for $p \nmid \disc(D)$ by \cite[Proposition 30.5.3 (a)]{Voight}.
Together with \cref{lem:local_emb},   this indicates $\fP_{\fo}(\phi):=\fP_{\iota,  \fo}(\phi)$ is independent of $\iota$.
Fix a set of representatives  $\{[\iota_j]\}_j$ of $\Cl(\fo)$-orbits in $\coprod_{[I]\in\Cl(\cO)}\Emb(\fo,  \cO(I))_{/\sim}$.
For $[\iota_j]$,  take $[I_j]\in\Cl(\cO)$ and $g_j\in B(\cO(I_j))_\Q$ so that $[\iota_j]\in\Emb(\fo,  \cO(I_j))_{/\sim}$ and $\iota_j=g_j^{-1}\cdot\iota_0$.
We set $\fP_{\iota_j,  \fo}(\phi)=\fP_{\iota,  \fo}(R_G(g_jx_{[I_j]})\phi)$ to obtain
    \[
    \sum_{j=1}^{|T_\A\bs B(\cO)_\A/K|}\fP_{\iota_j,  \fo}(\phi)
    =\frac{1}{u_\fo}\sum_{[I]\in\Cl(\cO)}\#\Big(\Emb(\fo,  \cO(I))_{/\sim}\Big)
    \cdot\phi([I]).
    \]
The same argument as the proof of \cref{thm:Fourier} (1) shows $\sum_{j=1}^{|T_\A\bs B(\cO)_\A/K|}\fP_{\iota_j,  \fo}(\phi)=c(E)\fP_{\fo}(\phi)$,  where $c(E)$ is defined in \cref{cor:Fourier}.
Hence \eqref{eq:Fourier-3} becomes
\begin{equation}\label{eq:Fourier-5}
    a_\phi(m^2|\Delta_E|)=c(E)\sum_{\substack{d \mid m \\ (d,  \disc(D))=1}}
    \fP_{\fo(d^2\Delta_E)}(\phi).
\end{equation}

Now we assume $E$ has an optimal embedding with respect to $\cO$ replacing $\cO$ and $\phi$ if necessary.
Fix $\iota_0\in\Emb(\fo_E,  \cO)$.
We also assume $\fP_\fo(\phi)\neq0$,  where $\fo=\fo(d^2\Delta_E)$.
Then $\cP_E\not\equiv0$ on $\pi$ and hence $\fP_E(\phi)\neq0$ by \cref{prop:Fourier}.
For each $p \nmid \disc(D)$,  we fix an isomorphism of $G_p$ with $\PGL_2(\Q_p)$ which sends $K_p$ to $\PGL_2(\Z_p)$.
Let $g(d)_p$ be the element in $G_p$ corresponding to $\diag(p^{\ord_p(d)},  1)$,  where $\ord_p$ is the usual $p$-adic valuation.
Then $g(d)_p^{-1}\cdot(\iota_0\otimes id_{\Q_p})$ is in $\Emb(\fo_p,  \cO_p)$.
Hence $\fP_\fo(\phi)=h_E^{-1}h_\fo\fP_{\iota_0,  E}(R_G(g(d))\phi)$ with $g(d)=(g(d)_v)_v\in G_{\A_f}$.
Since $\phi$ is a Hecke eigenform,  it is decomposable as $\phi=\otimes_v\phi_v$ with $\phi_v\in\pi_v$.
For a finite place $v$,  set $\beta_{\iota_0,  E_v}(d; \phi_v)=\alpha_{\iota_0,  E_v}(\phi_v)^{-1}\int_{T_v}\langle \pi_v(t_vg(d)_v)\phi_v,  \phi_v \rangle_v \d t_v$. Note that $\alpha_{\iota_0,  E_v}(\phi_v)\neq0$ as we have seen in the proof of \cref{prop:Fourier}.
Then we have
    \[
    \fP_\fo(\phi)=\frac{h_E^{-1}h_\fo}{u_E^{-1}u_\fo}\fP_E(\phi)
    \prod_{p \mid d}\beta_{\iota_0,  E_p}(d; \phi_p).
    \]
According to \cite[Theorem 6.7.2]{Miyake},   the class number $h_\fo$ is given as 
    \[
    h_E^{-1}h_\fo=\frac{d}{u_Eu_\fo^{-1}}
    \prod_{p \mid d}\left(1-\left(\frac{\Delta_E}{p}\right)p^{-1}\right).
    \]
Substituting these two equations into \eqref{eq:Fourier-5} and using \cref{cor:Fourier},  we obtain at least formally
\begin{align*}
    \sum_{m=1}^\infty\frac{a_\phi(m^2|\Delta_E|)}{m^s}
    &=a_\phi(|\Delta_E|)\sum_{m=1}^\infty 
    \sum_{\substack{d \mid m \\ (d,  \disc(D))=1}}
    m^{-s}d \prod_{p \mid d}\left(1-\left(\frac{\Delta_E}{p}\right)p^{-1}\right)
    \beta_{\iota_0,  E_p}(d; \phi_p) \\
    &=a_\phi(|\Delta_E|)\sum_{k=1}^\infty k^{-s}
    \sum_{(d,  \disc(D))=1}
    d^{-s+1} \prod_{p \mid d}\left(1-\left(\frac{\Delta_E}{p}\right)p^{-1}\right)
    \beta_{\iota_0,  E_p}(d; \phi_p) \\
    &=a_\phi(|\Delta_E|)\zeta_\fin(s)
    \prod_{p \nmid \disc(D)}\left\{
    1+\left(1-\left(\frac{\Delta_E}{p}\right)p^{-1}\right)
    \sum_{r=1}^\infty \beta_{\iota_0,  E_p}(p^r; \phi_p)p^{(-s+1)r}\right\}.
\end{align*}
Here we set $m=kd$ for the second equality,  $\zeta_\fin(s)=\prod_{v<\infty}\zeta_v(s)$ is the finite part of $\zeta(s)$ and we used $\beta_{\iota_0,  E_p}(d; \phi_p)=\beta_{\iota_0,  E_p}(p^{\ord_p(d)}; \phi_p)$.
The remaining problem is to show
\begin{equation}\label{eq:Hecke-2}
    1+\left(1-\left(\frac{\Delta_E}{p}\right)p^{-1}\right)
    \sum_{r=1}^\infty \beta_{\iota_0,  E_p}(p^r; \phi_p)p^{(-s+1)r}
    =\frac{L(s-\tfrac12,  \pi_p)}{L(s,  \eta_{E_p})\zeta_p(s)}
\end{equation}
for $p \nmid \disc(D)$,  where $L(s,  \pi_p)$ and $L(s,  \eta_{E_p})=\left(1-\left(\frac{\Delta_E}{p}\right)p^{-s}\right)^{-1}$ are local factors of $L(s,  \pi)$ and $L(s,  \eta_E)$,  respectively.
The convergence of the above infinite product over $p \nmid \disc(D)$ follows once we prove this equality.
Note that we have $L(s-\frac12,  \pi_p)=\zeta_p(s)$ for $p \mid \disc(D)$.

Let $p$ be a prime such that $p \nmid \disc(D)$ and $\lambda_p\in F_\pi$ the $p$-th Hecke eigenvalue of $\phi$,  \textit{i.e.}\,$T_p\phi=\lambda_p\phi$.
Then we have $L(s-\frac12,  \pi_p)=(1-\lambda_pp^{-s}+p^{-2s+1})^{-1}$.
By \cite[Proposition 3.4]{KP},  we have
    \[
    \sum_{r=1}^\infty \beta_{\iota_0,  E_p}(p^r; \phi_p)x^r=
    \frac{\beta_{\iota_0,  E_p}(p; \phi_p) x-p^{-1}x^2}
    {1-p^{-1}\lambda_p x+p^{-1}x^2},  \quad 
    \beta_{\iota_0,  E_p}(p;\phi_p)=
    \frac{\lambda_p-1-\left(\frac{\Delta_E}{p}\right)}
    {p-\left(\frac{\Delta_E}{p}\right)}
    \]
as a formal power series.
Substituting $x=p^{-s+1}$,  we obtain \eqref{eq:Hecke-2}.
This completes the proof.
\end{proof}

\begin{prop}\label{prop:Hecke}
Let $\phi\in\cS_N^\new(\cO)$ be a Hecke eigenform and $\pi$ the irreducible cuspidal automorphic representation of $G_\A$ generated by $\phi$.
Then
    \[
    D(s,  \cW(\phi))=(2\pi)^{-s}\Gamma(s)L_\fin(2s-\tfrac12,  \pi)
    \sum_{E\in X(D)}
    \frac{c(E)\fP_E(\phi)}{L^{S_\cO}(2s,  \eta_E)|\Delta_E|^s}.
    \]
\end{prop}

\begin{proof}
This is an immediate consequence of \cref{cor:Fourier},  \eqref{eq:Hecke} and \cref{lem:Hecke}.
\end{proof}

\subsection{Proof of \cref{thm:sign_change}}
\label{subsec:ProofSign}

Now we are ready to prove \cref{thm:sign_change}.

\begin{proof}[Proof of \cref{thm:sign_change}]
(1) From \cite[Theorem 1.5]{SW},  $\cP_E\not\equiv0$ on $\pi$ for infinitely many $E\in X(D)$.
The assertion follows from \cref{prop:Fourier}.

(2) Suppose for a contradiction that $\{\fP_E(\phi)\}_{E\in X(D)}$ has finitely many sign changes.
We may assume $\fP_E(\phi)<0$ for only finitely many $E$ by replacing $\phi$ with $-\phi$ if necessary.
Set 
    \[
    D_\pm(s,  \cW(\phi))=(2\pi)^{-s}\Gamma(s)L_\fin(2s-\tfrac12,  \pi)
    \sum_{\substack{E\in X(D) \\ \pm\fP_E(\phi)>0}}
    \frac{c(E)\fP_E(\phi)}{L^{S_\cO}(2s,  \eta_E)|\Delta_E|^s}.
    \]
Let $\zeta^{S_\cO}(s)=\prod_{p \nmid \disc(D)}\zeta_p(s)$ be a partial Dedekind zeta function.
Then we have 
    \[
    \frac{\zeta^{S_\cO}(s)}{L^{S_\cO}(s,  \eta_E)}
    =\sum_{(d,  \disc(D))=1}\frac{b_E(d)}{d^s},   \qquad 
    b_E(d):=\sum_{k \mid d}\eta_E(k)\mu(k)=\prod_{p \mid d}(1-\eta_E(p))\geq0,
    \]
where $\eta_E$ is seen as a Dirichlet character on $\Z/\Delta_E\Z$ and $\mu$ is the M\"obius function.
Thus we get
\begin{equation}\label{eq:positive_series}
    \frac{(2\pi)^s\zeta^{S_\cO}(2s)}{\Gamma(s)L_\fin(2s-\frac12,  \pi)}
    D_+(s,  \cW(\phi))=\sum_{\substack{E\in X(D) \\ \fP_E(\phi)>0}}
    \sum_{(d,  \disc(D))=1}\frac{b_E(d)c(E)\fP_E(\phi)}{d^{2s}|\Delta_E|^s}.
\end{equation} 
From \cite[Theorem 4]{SWY},  we obtain
    \[
    \sum_{E\in X(D), \, |\Delta_E|<x}|\fP_E(\phi)| \geq 
    \left(\sum_{E\in X(D), \, |\Delta_E|<x }|\fP_E(\phi)|^2\right)^\frac12 
    \gg x^\frac34
    \]
when $x\to\infty$.
Hence
    \[
    \left|\sum_{\substack{E\in X(D) \\ \fP_E(\phi)>0}}\sum_{(d,  \disc(D))=1}
    \frac{b_E(d)c(E)\fP_E(\phi)}{d^{2s}|\Delta_E|^s}\right|
    \geq \sum_{E\in X(D), \, |\Delta_E|<x}\frac{|\fP_E(\phi)|}{|\Delta_E|^{\re(s)}}
    \gg x^{\frac34-\re(s)}
    \]
 for $\re(s)>0$.
 Since $b_E(d)\ge0$,  this indicates that the right had side of \eqref{eq:positive_series} has a pole at $s=s_0$ for some $s_0\geq\frac34$.
 On the other hand,  since $D_-(s,  \cW(\phi))$ and $D(s,  \cW(\phi))=D_+(s,  \cW(\phi))+D_-(s,  \cW(\phi))$ are holomorphic for $\re(s)>\frac34$,  so is $D_+(s,  \cW(\phi))$.
It is also known that $L_\fin(1,  \pi)\neq0$.
Combining these facts,  the right hand side of \eqref{eq:positive_series} does not have poles for $\re(s)\geq\frac34$,  which is a contradiction.

(3) Suppose that $\{\fP_E(\phi)^{(i)}\}_{E\in X(D)}$ has only finitely many sign changes.
For each $i$,  take $\delta_i\in\{\pm1\}$ so that $\delta_i\cdot\fP_E(\phi)^{(i)}<0$ for only finitely many $E$ and set $\varphi=\sum_{i=1}^{[F_\pi : \Q]}\delta_i\cdot\phi^{(i)}\in\cS_N^\new(\cO)$.
Then except for finite number of $E\in X(D)$,
    \[
    \fP_E(\varphi)=\sum_{i=1}^{[F_\pi : \Q]}|\fP_E(\phi)^{(i)}|
    \geq \frac{1}{\ds\max_{1\leq i\leq [F_\pi:\R]} |\iota_\pi(v_i)|}
    \sum_{i=1}^{[F_\pi : \Q]}\left|\iota_\pi\left(\fP_E(\phi)^{(i)}v_i\right)\right|
    \geq \frac{1}{\ds\max_{1\leq i\leq [F_\pi:\R]}|\iota_\pi(v_i)|} |\fP_E(\phi)|.
    \]

As we remarked at the beginning of this section,  each $\phi^{(i)}$ is in $\mathrm{Span}_\R\{\iota\circ\phi \mid \iota\in\Hom(F_\pi,  \R)\}$.
Thus we can write $\varphi=\sum_\iota a_\iota \cdot(\iota\circ\phi)$ with some $a_\iota\in\R$,  where $\iota$ runs over $\Hom(F_\pi,  \R)$.
Note that $\iota\circ\phi$ is a Hecke eigenform.
Let $\pi^\iota$ be the irreducible cuspidal automorphic representation of $G_\A$ generated by $\iota\circ\phi$.
From \cref{prop:Hecke},  we obtain
    \begin{equation}\label{eq:lin_comb}
    \sum_{\iota\in\Hom(F_\pi,  \R)} a_\iota 
    \frac{(2\pi)^sD(s,  \cW(\iota\circ\phi))}
    {\Gamma(s)L_\fin(2s-\tfrac12,  \pi^\iota)}=\sum_{E\in X(D)}
    \frac{c(E)\fP_E(\varphi)}{L^{S_\cO}(2s,  \eta_E)|\Delta_E|^s}.
    \end{equation}
The left hand side of \eqref{eq:lin_comb} is holomorphic for $\re(s)>\tfrac34$.
On the other hand,  the above lower estimate for $\fP_E(\varphi)$ and the same argument as in the proof of (2) show that the right hand side of \eqref{eq:lin_comb} has a pole for $\re(s)\geq\tfrac34$.
This is a contradiction.
\end{proof}


\section{Conjectures}
\label{sec:Conjectures}

Based on numerical experiments,  we formulate a conjecture on the distribution of $\{\fP_E(\phi)\}_E$ for $\phi\in\cS^\new(\cO)$.
We also discuss their relation with the central limit conjecture of \cite{CKRS}.

We fix a normalized Hecke eigenform $\phi\in\cS_N^\new(\cO)$.
Let $\pi$ be the irreducible cuspidal automorphic representation of $G_\A$ generated by $\phi$ and $\pi'=\otimes_v\pi'_v$ its Jacquet-Langlands transfer to $\PGL_2(\A)$.
Set $Y(D; \pi)=\bigcup_{\{\cE_v\}_{v\in S_\cO}}X(\{\cE_v\}_{v\in S_\cO})$,  where $\{\cE_v\}_{v\in S_\cO}$ runs over elements in $\prod_{v\in S_\cO}X_v$ such that the ramification set of $D$ coincides with $\{v\in S_\cO \mid \varepsilon(\pi'_v; \cE_v)=-1\}$.
Then $Y(D; \pi)$ is a subset of $X(D)$ and the period $\cP_E$ vanishes on $\pi$ for any $E\in X(D)\setminus Y(D; \pi)$,  according to \cref{thm:Waldspurger} (1).
By \cref{cor:optimal-3},  all but finitely many $E\in Y(D,  \pi)$ has an optimal embedding with respect to $\cO$.
We focus on the distribution of  $\{\fP_E(\phi)\}_{E\in Y(D; \pi)}$ and for that purpose,  we may ignore finitely many $E\in Y(D,  \pi)$'s which do not have optimal embeddings with respect to $\cO$.

\subsection{Symmetry conjecture}
\label{subsec:Symmetry conjecture}

Let $\fo_\pi$ denote the ring of integers of the Hecke field $F_\pi$ of $\phi$.
Recall that $\phi$ takes values in $\fo_\pi$ since it is assumed to be normalized.
For $x>0$ and $z\in\fo_\pi$,  set 
\begin{equation}\label{eq:probability}
    \bP_x[\fP(\phi)=z]
    =\frac{\#\{E\in Y(D; \pi) \mid |\Delta_E|<x,  \ \fP_E(\phi)=z\}}
    {\#\{E\in Y(D; \pi) \mid |\Delta_E|<x\}}.
\end{equation}
This is the probability that the random variable $\fP(\phi)\,\colon\{E\in Y(D; \pi) \mid |\Delta_E|<x\} \rightarrow \fo_\pi$ which sends $E$ to $\fP_E(\phi)$ takes the value $z\in\fo_\pi$.
Let $\bE_x[\fP(\phi)]$ be the first moment,  which coincides with the average of the periods $\{\fP_E(\phi)\}_E$: 
\begin{align*}
\bE_x[\fP(\phi)]&=\sum_{z\in\fo_\pi}z\cdot \bP_x[\fP(\phi)=z] \\
    &=\sum_{\substack{E\in Y(D; \pi) \\ |\Delta_E|<x}}
    \frac{\fP_E(\phi)}{\#\{E\in Y(D; \pi) \mid |\Delta_E|<x\}}.
\end{align*}

From several numerical experiments,  we predict that for any $z\in\fo_\pi$,  the probability for $\fP(\phi)=z$ is ``almost the same'' as that for $\fP(\phi)=-z$ in the following sense:

\begin{conj}\label{conj:symmetry}
For $x>0$ and a sufficiently small $\delta>0$,  
\begin{equation}\label{eq:symmetry}
    \frac12\sum_{z\in\fo_\pi}\Big|\bP_x[\fP(\phi)=z]-\bP_x[\fP(\phi)=-z]\Big| \ll 
    \frac{x^{1-\delta}}{\#\{E\in Y(D; \pi) \mid |\Delta_E|<x\}}.
\end{equation}
\end{conj}
Brunier and Kohnen \cite{BK} suggested that half of the Fourier coefficients $a(d)$ of a half-integral weight modular form $h$ are positive when $d$ ranges over fundamental discriminants with $a(d)\neq0$.
Combined with \cref{cor:Fourier},  \cref{conj:symmetry} refines their conjecture for $h=\cW(\phi)$.

Since $\#\{E\in Y(D; \pi) \mid |\Delta_E|<x\}\sim C\cdot x$ for some positive constant $C$,  from \eqref{eq:symmetry} we obtain
\begin{align*}
    \Big|\bE_x[\fP(\phi)]\Big|
    &=\frac12\, \Big|
    \sum_{z\in\fo_\pi}z\cdot\{\bP_x[\fP(\phi)=z]-\bP_x[\fP(\phi)=-z]\}\Big| \\
    &\ll x^{-\delta}\cdot \max\{|z| \mid z\in\fo_\pi,  \ \bP_x[\fP(\phi)=z]\neq0\}.
\end{align*}
Here,  $\bE_x[\fP(\phi)]$ and elements in $\fo_\pi$ are treated as real numbers under a fixed embedding $F_\pi \hookrightarrow \R$.

\subsection{Numerical examples}
\label{subsec:Numerical}

In what follows,  we present numerical examples for \cref{conj:symmetry} in several cases using \texttt{Magma} \cite{BCP}.
By \cref{cor:Fourier},  our computation is essentially the same as that of \cite{Hamieh},  which provides a Sage package to compute the classical Waldspurger's lift.

\example\label{ex:11.2.a.a}
Suppose $\disc(D)=\disc(\cO)=11$.
Since $\dim\cS(\cO)=1$ and $F_\cO=\Q$,  a normalized Hecke eigenform $\phi\in\cS(\cO)$ takes values in $\fo_\pi=\Z$.
Write $\Cl(\cO)=\{[I_1],  [I_2]\}$ where $I_1=\cO$.
Then we may choose $\phi$ so that $\phi([I_1])=2$ and $\phi([I_2])=-3$.
We computed $\fP_E(\phi)$ for $E\in Y(D; \pi)$ with $|\Delta_E|<10^6$. 
There are 164,511 such $E$'s and the result is shown in Figure \ref{fig:11.2.a.a}. 
The horizontal axis represents $z\in\Z$ and the vertical axis represents $\#\{E\in Y(D; \pi) \mid |\Delta_E|<10^6,  \ \fP_E(\phi)=z\}$,  the numerator of \eqref{eq:probability}.
One can find the same graph with the precise value of each point in \cite{SWY22}.
\begin{figure}[h]
\includegraphics[height=9cm, width=15cm]{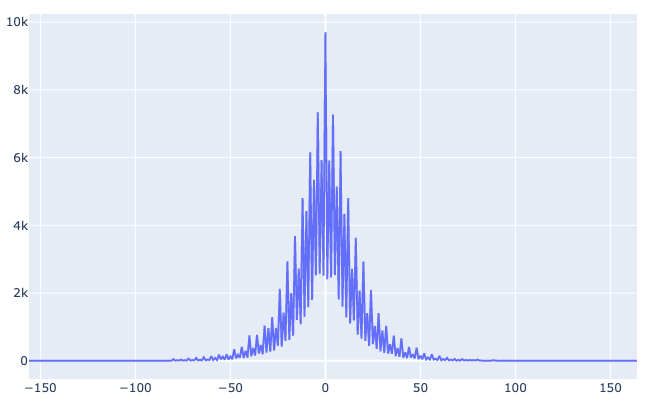}\vspace{-15pt}

\caption{}\label{fig:11.2.a.a}
\end{figure}

Table \ref{table:symmetry:11} shows the values of both sides of \eqref{eq:symmetry} in this case for $x=i\cdot10^5$,  $i=2, 4,6,8,10$.

\begin{table}[htb]\renewcommand{\arraystretch}{1.3}
  \begin{tabular}{|c|c|c|c|c|c|} \hline
$x$ & $2\cdot10^5$ & $4\cdot10^5$ & $6\cdot10^5$ & $8\cdot10^5$ & $10^6$\\ \hline
  LHS of \eqref{eq:symmetry} & 
  0.027836 & 0.019179 & 0.017544 & 0.013720 & 0.012166 \\ \hline
  \end{tabular}\vspace{3pt}

\caption{}\label{table:symmetry:11}
\end{table}

\example\label{ex:23.2.a.a}
Suppose $\disc(D)=\disc(\cO)=23$.
In this case,  $\dim\cS(\cO)=2$ and $F_\cO=\Q(\sqrt{5})$.
A normalized Hecke eigenform takes values in $\fo_\pi=\Z[\frac{1+\sqrt{5}}{2}]$. 
We realize $D$ and $\cO$ as 
\begin{align*}
D&=\Q+\Q x+\Q y+ \Q z,  \quad x^2=-23,  \, y^2=-1,  z=xy=-yx,  \\
\cO&=\Z+\Z y+\Z(\tfrac12 y+\tfrac12 z)+\Z(\tfrac12+\tfrac12 x).
\end{align*}
Then right fractional $\cO$-ideals $I_1,  I_2,  I_3$ given as
\begin{align*}
I_1=\cO,  \ I_2=2\Z+2\Z y+\Z(\tfrac12 y+\tfrac12 z)+\Z(\tfrac12-\tfrac12 x),  \\
I_3=3\Z+3\Z y+\Z(2-\tfrac32 y-\tfrac12 z)+\Z(\tfrac32-\tfrac12 x-y)
\end{align*}
form a set of representatives of $\Cl(\cO)$.
Then we may choose $\phi$ so that $\phi([I_1])=-1-\sqrt{5}$,  $\phi([I_2])=\frac{-1+\sqrt{5}}{2}$ and $\phi([I_3])=3$.
We computed $\fP_E(\phi)$ for $E\in Y(D; \pi)$ with $|\Delta_E|<10^6$,  there are 157,925 such $E$'s.
The Figure \ref{fig:23.2.a.a} is the graph of the numerator of \eqref{eq:probability}.
If we write $z=a+b\cdot\frac{1+\sqrt{5}}{2}\in\fo_\pi$ with $a, b\in\Z$,  then the horizontal axis represents $a$ and the vertical axis represents $b$.
One can find the same graph with the precise value of each point in \cite{SWY22}.
\begin{figure}[h]
\includegraphics[height=9cm, width=15cm]{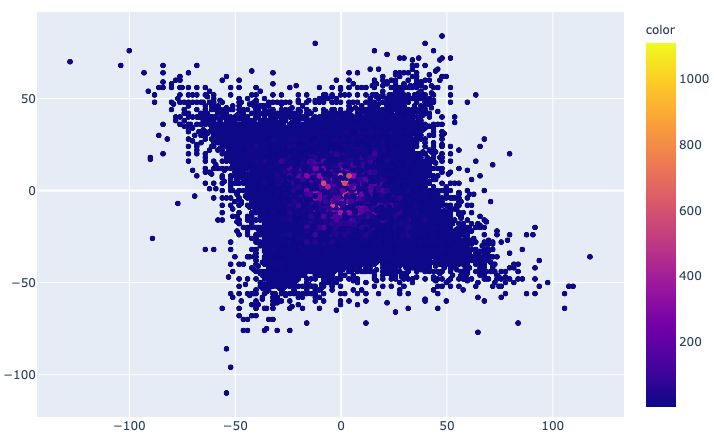}\vspace{-15pt}

\caption{}\label{fig:23.2.a.a}
\end{figure}

Table \ref{table:symmetry:23} shows the values of both sides of \eqref{eq:symmetry} in this case for $x=i\cdot10^5$,  $i=2, 4,6,8,10$.

\begin{table}[htb]\renewcommand{\arraystretch}{1.3}
  \begin{tabular}{|c|c|c|c|c|c|} \hline
$x$ & $2\cdot10^5$ & $4\cdot10^5$ & $6\cdot10^5$ & $8\cdot10^5$ & $10^6$\\ \hline
  LHS of \eqref{eq:symmetry} & 
 0.10918 & 0.093252 & 0.086587 & 0.080091 & 0.077341 \\ \hline
  \end{tabular}\vspace{3pt}

\caption{}\label{table:symmetry:23}
\end{table}

\example\label{ex:41.2.a.a}
Suppose $\disc(D)=\disc(\cO)=41$.
In this case,  $\dim\cS(\cO)=3$ and $F_\cO$ is the splitting field of $f(t)=t^3-t^2-3t+1$.
Let $\alpha$ be one of the roots of  $f(t)$. 
Then $\fo_\pi=\Z[\alpha]$.
We realize $D$ and $\cO$ as 
\begin{align*}
D&=\Q+\Q x+\Q y+ \Q z,  \quad x^2=-41,  \, y^2=-3,  z=xy=-yx,  \\
\cO&=\Z+\Z (\tfrac12+\tfrac12 y)+\Z(\tfrac12+\tfrac12 x+\tfrac16 y+\tfrac16 z)+\Z(\tfrac12-\tfrac16 y+\tfrac13 z). 
\end{align*}
Then right fractional $\cO$-ideals $I_1,  I_2,  I_3,  I_4$ given as
\begin{align*}
I_1=\cO,  \ I_2=2\Z+\Z(1-y)+\Z(\tfrac12-\tfrac12 x-\tfrac16 y-\tfrac16 z),
+\Z(1-\tfrac23 y+\tfrac13 z),  \\
I_3=3\Z+\Z(\tfrac32 -\tfrac32 y)+\Z(\tfrac52 -\tfrac12 x-\tfrac76 y-\tfrac16 z)+\Z(\tfrac12-\tfrac76 y +\tfrac13 z), \\
I_4=5\Z+\Z(\tfrac52 -\tfrac52 y)+\Z(\tfrac72 -\tfrac12 x - \tfrac76 y - \tfrac16 z)+\Z(\tfrac52 -\tfrac16 y +\tfrac13 z)
\end{align*}
form a set of representatives of $\Cl(\cO)$.
Then we may choose $\phi$ so that $\phi([I_1])=3\alpha$,  $\phi([I_2])=1-\alpha-\alpha^2$,  $\phi([I_3])=-1$ and $\phi([I_4])=\alpha^2$.
We computed $\fP_E(\phi)$ for $E\in Y(D; \pi)$ with $|\Delta_E|<5\cdot10^5$,  there are 77,035 such $E$'s.
One can find the 3D graph of $\#\{E\in Y(D; \pi) \mid |\Delta_E|<5\cdot10^5,  \ \fP_E(\phi)=z\}$ in \cite{SWY22},  where the point $z=a+b\cdot\alpha+c\cdot\alpha^2\in\fo_\pi$ is plotted as $(a,b,c)$. 

Table \ref{table:symmetry:41} shows the values of both sides of \eqref{eq:symmetry} in this case for $x=i\cdot10^5$,  $i=1, 2,3,4,5$.
\begin{table}[htb]\renewcommand{\arraystretch}{1.3}
  \begin{tabular}{|c|c|c|c|c|c|} \hline
$x$ & $10^5$ & $2\cdot10^5$ & $3\cdot10^5$ & $4\cdot10^5$ & $5\cdot10^5$\\ \hline
  LHS of \eqref{eq:symmetry} & 
 0.24973 & 0.23801 & 0.22803 & 0.22000 & 0.21215 \\ \hline
  \end{tabular}\vspace{3pt}

\caption{}\label{table:symmetry:41}
\end{table}

\subsection{Central limit conjecture}
\label{subsec:Central limit conjecture}

There are a lot of conjectures which predict that the value distribution of certain families of $L$-functions is described by random matrix theory.
They are stated in the form of a central limit theorem,  so we call them \textit{central limit conjectures}.
In \cite{CKRS},  Conrey,  Keating,  Rubinstein and Snaith combined those conjectures with well-known formulas that relate the central values of $L$-functions associated with elliptic curves with the Fourier coefficients of half-integral weight modular forms. 
As a consequence,  they formulated a central limit conjecture of the value distribution of the Fourier coefficients.

Since the periods of algebraic modular forms are the Fourier coefficients of modular forms of weight $\frac32$,  \cite[Conjecture 4.1]{CKRS} leads to the following central limit conjecture for $\fP_E(\phi)$.

\begin{conj}\label{conj:central_limit-1}
For $\alpha,  \beta\in\R\cup\{\pm\infty\}$ with $\alpha<\beta$,  
\begin{align*}
\lim_{x\to\infty}&\frac{1}{\#\{E\in Y(D; \pi) \mid |\Delta_E|<x\}} \\
&\quad\times\#\left\{E\in Y(D; \pi)  \, \middle|\, |\Delta_E|<x, \, \fP_E(\phi)\neq0 
\text{ and }
\frac{2\log|\fP_E(\phi)|-\frac12\log\left(\frac{|\Delta_E|}{\log|\Delta_E|}\right)}
{(\log\log|\Delta_E|)^\frac12}\in(\alpha,  \beta)\right\} \\
&\hspace{100pt}=\frac{1}{\sqrt{2\pi}}
\int_\alpha^\beta\exp\left(-\frac{t^2}{2}\right)\d t.
\end{align*}
Here,  $\fP_E(\phi)$ is treated as a real number under a fixed embedding $F_\pi \hookrightarrow \R$.
\end{conj}

Further,  we obtain the following conjecture combining \cref{conj:central_limit-1} with \cref{conj:symmetry}.

\begin{conj}\label{conj:central_limit-2}
For $\alpha,  \beta\in\R\cup\{\pm\infty\}$ with $\alpha<\beta$ and a sign $\kappa\in\{\pm1\}$,  
\begin{align*}
\lim_{x\to\infty}&\frac{1}{\#\{E\in Y(D; \pi) \mid |\Delta_E|<x\}}\\
&\quad\times\#\left\{E\in Y(D; \pi)  \, \middle|\, |\Delta_E|<x, \,  \kappa\cdot\fP_E(\phi)>0 
\text{ and }
\frac{\log|\fP_E(\phi)|-\frac14\log\left(\frac{|\Delta_E|}{\log|\Delta_E|}\right)}
{(\log\log|\Delta_E|)^\frac12}\in(\alpha,  \beta)\right\} \\
&\hspace{100pt}=\frac{1}{2\sqrt{2\pi}}
\int_\alpha^\beta\exp\left(-\frac{t^2}{2}\right)\d t.
\end{align*}
Here,  $\fP_E(\phi)$ is treated as a real number under a fixed embedding $F_\pi \hookrightarrow \R$.
\end{conj}

\cref{conj:central_limit-2} with $\alpha=-\infty$ and $\beta=0$ reads
\begin{align*}
\lim_{x\to\infty}&\frac{1}{\#\{E\in Y(D; \pi) \mid |\Delta_E|<x\}}\\
&\quad\times\#\left\{E\in Y(D; \pi)  \, \middle|\,  |\Delta_E|<x, \, \kappa\cdot\fP_E(\phi)>0 
\text{ and }
|\fP_E(\phi)|<\left|\frac{\Delta_E}{\log|\Delta_E|}\right|^\frac14\right\} 
=\frac14.
\end{align*}
On the other hand,  \cref{conj:central_limit-2} with $\alpha=r>0$ and $\beta=\infty$ reads
\begin{align*}
\lim_{x\to\infty}&\frac{1}{\#\{E\in Y(D; \pi) \mid |\Delta_E|<x\}}\\
&\quad\times\#\left\{E\in Y(D; \pi)  \, \middle|\,  |\Delta_E|<x, \, \kappa\cdot\fP_E(\phi)>0 
\text{ and }
|\fP_E(\phi)|>\left|\frac{\Delta_E}{\log|\Delta_E|}\right|^\frac14
\exp(r\sqrt{\log\log|\Delta_E|})\right\} \\
&\hspace{100pt} =\frac14 \mathrm{erfc}(\sqrt{2}r),
\end{align*}
where 
    \[
    \mathrm{erfc}(r)=\sqrt{\frac{2}{\pi}}\int_r^\infty 
    \exp\left(-\frac{t^2}{2}\right)\d t
    \] 
is the complementary error function.
Roughly speaking,  this observation indicates that the values of the periods $\fP_E(\phi)$ are concentrated around 0.
This matches with the numerical experiments we see in \cref{ex:11.2.a.a},  \cref{ex:23.2.a.a} and \cref{ex:41.2.a.a}.

\end{document}